\documentclass[11pt,reqno]{amsart}

\usepackage{amsmath, amsfonts, amssymb, amscd, enumerate, young, youngtab, empheq, stmaryrd, slashed}
\usepackage{tikz}
\usepackage{hyperref}
\usepackage{here}
\usetikzlibrary{arrows,decorations.markings,positioning}
\tikzset{inner sep=0pt, node distance=5mm,
  root/.style={circle,draw,minimum size=5pt,thick},
  broot/.style={circle,draw,minimum size=5pt,thick,fill},
  xroot/.style={circle,draw,minimum size=5pt,thick,label=below:$\times$},
	iroot/.style={circle,draw,minimum size=5pt,thick,label=above:{\tiny$\alpha_i$}},
	jroot/.style={circle,draw,minimum size=5pt,thick,label=above:{\tiny$\alpha_j$}},
	kroot/.style={circle,draw,minimum size=5pt,thick,label=above:{\tiny$\alpha_k$}},
  doublearrow/.style={postaction={decorate},   decoration={markings,mark=at position .6 with {\arrow[line width=1.2pt]{>}}},double distance=1.6pt,thick},
  rdoublearrow/.style={postaction={decorate},   decoration={markings,mark=at position .4 with {\arrowreversed[line width=1.2pt]{>}}},double distance=1.6pt,thick},
  curvedline/.style={bend=right}
	}
\theoremstyle{plain}
\newtheorem*{theorem*}{Theorem}
\newtheorem{theo}{Theorem}[section]

\theoremstyle{definition}

\newtheorem{definition}[theo]{Definition}

\theoremstyle{plain}
\newtheorem{lemma}[theo]{Lemma}
\newtheorem{theorem}[theo]{Theorem}

\newtheorem{proposition}[theo]{Proposition}

\theoremstyle{definition}

\newcommand{\beq}{\begin{equation}}
\newcommand{\eeq}{\end{equation}}

\renewcommand{\o}{\omega}

\renewcommand{\[}{\llbracket}

\newcommand{\res}{\operatorname{res}_W}
\newcommand{\tE}{\textbf{E}}
\newcommand{\tH}{\textbf{H}}
\newcommand{\tV}{\textbf{V}}

\newcommand{\bC}{\mathbb{C}}

\newcommand{\bR}{\mathbb{R}}
\newcommand{\bZ}{\mathbb{Z}}

\newcommand{\bH}{\mathbb{H}}

\newcommand{\bN}{\mathbb{N}}

\newcommand{\tC}{\textbf{C}}
\newcommand{\tD}{\textbf{D}}

\renewcommand{\gg}{\mathfrak{g}}

\newcommand{\gm}{\mathfrak{m}}

\newcommand{\gs}{\mathfrak{s}}
\newcommand{\gt}{\mathfrak{t}}

\newcommand{\so}{\mathfrak{so}}

\newcommand{\ggl}{\mathfrak{gl}}
\declareslashed{}{/}{0}{0}{\partial}

\newcommand\PGL{\mathrm{PGL}}
\newcommand\GL{\mathrm{GL}}

\newcommand\SO{\mathrm{SO}}

\newcommand\Sp{\mathrm{Sp}}
\renewcommand\sp{\mathfrak{sp}}
\renewcommand\sl{\mathfrak{sl}}

\newcommand{\cD}{\mathcal{D}}

\newcommand{\cR}{\mathcal{R}}

\newcommand{\cU}{\mathcal{U}}

\renewcommand{\square}{\kern1pt\vbox
{\hrule height 0.6pt\hbox{\vrule width 0.6pt\hskip 3pt
\vbox{\vskip 6pt}\hskip 3pt\vrule width 0.6pt}\hrule height0.6pt}\kern1pt}

\DeclareMathOperator\End{End\;}

\DeclareMathOperator\Ad{Ad}
\DeclareMathOperator\ad{ad}

\renewcommand\Im{\operatorname{Im}}
\newcommand\Ker{\operatorname{Ker}}

\newcommand{\Hom}{{\operatorname{Hom}}}

\newcommand{\wt}{\widetilde}

\newcommand{\be}{\begin{equation}}
\newcommand{\ee}{\end{equation}}

\def\<#1,#2>{\langle\,#1,\,#2\,\rangle}
\newcommand{\arr}{\begin{array}{rlll}}
\newcommand{\ea}{\end{array}}
\newcommand{\bea}{\begin{eqnarray}}
\newcommand{\eea}{\end{eqnarray}}
\newcommand{\bean}{\begin{eqnarray*}}
\newcommand{\eean}{\end{eqnarray*}}
\newcounter{ssig}
\newcounter{ttig}
\newcommand{\y}{\hskip0.1cm\tiny\young}
\title[Almost CR quaternionic manifolds]
{Almost CR quaternionic manifolds\\ and their immersibility in $\bH$P$^n$.}
\author{A. Santi}
\address{
Andrea Santi, Dipartimento di Matematica e Informatica, Universit\'a di
Parma, Parco Area delle Scienze 53/A, 43124, Parma, Italy}
\email{asanti.math@gmail.com, andrea.santi@unipr.it}
\thanks{{\it Acknowledgments}. This research was partially supported by the Project 
Firb 2012 {\it Geometria differenziale e teoria geometrica delle funzioni}, by GNSAGA
of INdAM and by project F1R-MTH-PUL-08HALO-HALOS08 of University of Luxembourg.}
\begin{document}
\newcommand{\bnabla}{\mbox{\boldmath$\nabla$}}
\begin{abstract} 
We apply the general theory of codimension one integrability conditions for $G$-structures developed in \cite{AS} to the case of quaternionic CR geometry.
We obtain necessary and sufficient conditions for an almost CR quaternionic manifold to admit local immersions as an hypersurface of the quaternionic projective space. We construct a deformation of the standard quaternionic contact structure on the quaternionic Heisenberg group which does not admit local immersions in any quaternionic manifold.
\end{abstract}
\maketitle
\null \vspace*{-.25in}
\section{Introduction}
\label{Introduction}
\setcounter{section}{1}
\setcounter{equation}{0}
The main aim of this paper is to apply the general theory of codimension one integrability conditions for $G$-structures developed in \cite{AS} to the case of quaternionic CR geometry. Let us first recall a familiar definition from complex CR geometry: a real manifold $M$ of dimension $2n-1$ is an {\it almost Cauchy-Riemann manifold (of hypersurface type)} if at each tangent space $T_xM$ there is a distinguished subspace $\cD_x\subset T_xM$ of real dimension $2n-2$ and a complex structure $I_x$ on it, both depending smoothly on $x\in M$ (the collection of subspaces $\cD_x$ constitutes a complex distribution $\cD\subset TM$).
\par
The notion of almost CR manifold arose in the study of real hypersurfaces of $\bC^n$. Indeed any such $M\subset \bC^n$  is endowed with the complex distribution $\cD$ given by $\cD_x=T_xM\cap iT_xM$ and satisfying the following additional property:
for any two sections $X,Y$ of $\cD$
\beq
\label{stufa1}
[X,Y]-[IX,IY]\in\Gamma(\cD)\qquad\text{and}\;\;\;\;\;\;\;\;\;\;\;\;
\eeq
\beq
\label{stufa2}
[IX,Y]+[X,IY]=I[X,Y]-I[IX,IY]\,.
\eeq
Equations \eqref{stufa1} and \eqref{stufa2} are necessary conditions for an almost CR manifold to admit local immersions in $\bC^n$ so that the almost CR structure is induced by the complex structure of the ambient space. It is well-known that they are also sufficient if $M$ is strictly pseudo-convex and $2n-1\geq 7$ (cf. \cite{Kur}) or if all the data are real-analytic (cf. \cite{AH}; see also \cite{AS} for an alternative proof using $G$-structures).
In this paper we wish to apply the general theory of $G$-structures, i.e. reductions $\pi:P\rightarrow M$ of the bundle of all linear frames on a manifold  $M$ (see \cite{St, Kb}), to the study of the analogous situation in quaternionic CR geometry. 
\vskip0.2cm\par
We recall that a real manifold $\wt M$ of dimension $4n$ is a {\it quaternionic manifold} if its tangent bundle is endowed with a linear quaternionic structure and a compatible torsion-free linear connection (see  \cite{Sal, AM}). The maximally homogeneous model for quaternionic manifolds is the quaternionic projective space $\bH$P$^n$, that is the compactification of $\bH^{n}$. 
On the other hand, manifolds endowed with a kind of quaternionic CR structure appeared for the first time in the $'00s$, introduced by O.\ Biquard in \cite{Bi2} and motivated by the study of conformal infinities of quaternionic K\"ahler metrics:
a real manifold $M$ of dimension $4n-1$ is {\it quaternionic contact} if it admits a distribution $\cD$ of rank $4n-4$ with the property that the symbol algebra $\gm(x)=\gm_{-2}(x)+\gm_{-1}(x)$ associated with the natural filtration $T^{-2}_xM=T_xM\supset T_x^{-1}M=\cD_x$ of $T_xM$ is at all points isomorphic with the
quaternionic Heisenberg algebra $\gm=\Im(\bH)+\bH^{n-1}$. 
In particular the distribution $\cD$ is  endowed with a linear quaternionic structure and a conformal class of Hermitian metrics (see \cite{AK2}).
\par
The maximally homogeneous model for quaternionic contact manifolds is the compactification $M=\Sp(n,1)/P$ of the quaternionic Heisenberg group,
where $P$ is the parabolic subgroup stabilizing an isotropic quaternionic line. This is a flag manifold and a real hypersurface of $\bH$P$^n$. 
Unfortunately not all real hypersurfaces of $\bH$P$^n$ are quaternionic contact (cf. \cite{Duc2}). To overcome this issue we consider here the following definition, which encompasses the notion of quaternionic contact manifold in a non-metrical framework. 
\begin{definition}\cite{MOP}\label{lavitafinisce} 
An {\it almost CR quaternionic manifold (of hypersurface type)} is 
a real manifold $M$ of dimension $4n-1$ together with 
a codimension one embedding $TM\rightarrow E$ of its tangent bundle $TM$ in a vector bundle $E$ endowed with a linear quaternionic structure.
\end{definition}
Let $\wt G=\GL_{n}(\bH)\cdot \Sp_1$ be the group of (twisted) linear automorphisms of the quaternionic vector space $V=\bH^n$ and $G_\sharp$ the subgroup 
of $\wt G$ which stabilizes the real subspace $W=\bH^{n-1}\oplus\Im\bH\subset V$ of real codimension one (see Lemma \ref{impor} for its explicit description).
The datum of an almost CR quaternionic structure is equivalent to a $G$-structure $\pi:P\rightarrow M$ on $M$ with structure group $G\subset\GL_{4n-1}(\bR)$ obtained upon restricting the action of $G_\sharp$ to $W$. 
On the other hand, as already advertised, the model ambient space for quaternionic CR geometry is $\wt M=\bH$P$^n$ and its quaternionic structure is an (integrable) $\wt G$-structure $\wt\pi:\wt P\to\wt M$.

It is a classical result (cf. \cite{G}) that obstructions to integrability for a $G$-structure can be expressed in terms of appropriate $G$-equivariant functions $\cR^{p+1}:P\rightarrow H^{p,2}(\gg)$ with values in the Spencer cohomology groups $H^{p,2}(\gg)$ of the Lie algebra $\gg=Lie(G)$ of $G$.
The general framework is exploited in \cite{Sal, AM}, where quaternionic manifolds are investigated and
the groups $H^{p,2}(\wt\gg)$, $\wt\gg=\ggl_{n}(\bH)\oplus\sp_1$, are explicitly computed (see also \cite{Oc2}). 
The analysis in this paper is based on a generalization of the obstruction theory to the case where the $G$-structure $\pi:P\to M$ under consideration is not integrable but possibly induced by an immersion into an ambient space $\wt M$ endowed with an integrable $\wt G$-structure $\wt\pi:\wt P\to\wt M$. 
The general theory of such {\it induced $G$-structures}  has been recently developed in \cite{AS} where obstructions to generalized integrability are expressed in terms of essential $(G,\wt G)$-curvatures, maps $\cR^{p+1}:P\rightarrow H^{p,2}(\wt\gg,W)$ with values in appropriate generalizations $H^{p,2}(\wt\gg,W)$ of the Spencer groups (see Definition \ref{RSCG}).
\vskip0.2cm\par
We now state the main result. Therein:
\begin{itemize}
\item[--] $\wt\gg=\ggl_{n}(\bH)\oplus \sp_1$ and $W=\bH^{n-1}\oplus\Im\bH$,
\item[--] $\gs\simeq \sl_{2n-2}(\bC)\oplus\sl_{2}(\bC)$ is the (complexified) Levi factor of $\gg$, 
\item[--] $\tE=\bC^{2n-2}$ and $\tH=\bC^{2}$ are the defining representations of $\sl_{2n-2}(\bC)$ and $\sl_{2}(\bC)$ respectively,
\item[--] $\Ad$ is the adjoint representation of $\sl_{2n-2}(\bC)$,
\item[--] $\tD$ is the irreducible representation of $\sl_{2n-2}(\bC)$ given by the kernel of the natural contraction $\tE\otimes \Lambda^2\tE^*\rightarrow \tE^*$.
\end{itemize}
\begin{theorem}\label{precedente}
Let $M$ be an almost CR quaternionic manifold of dimension $4n-1\geq 7$ and $\pi:P\rightarrow M$ its canonically associated $G$-structure. Then
there exists a canonical $G$-equivariant map $\cR^{1}:P\rightarrow H^{0,2}(\wt\gg, W)$ and, 
if $\cR^1$ identically vanishes, a canonical $G$-equivariant map $\cR^{2}:P\rightarrow H^{1,2}(\wt\gg,W)$.
If $M$ is locally immersible into a quaternionic manifold then $\cR^1= 0$. Moreover $\cR^1=\cR^2= 0$ if and only if
$M$ is locally immersible in $\bH${\rm P}$^n$ around any point, in such a way that the almost CR quaternionic structure is induced by the quaternionic structure of the ambient space. Finally there exist natural $\gs$-equivariant isomorphisms
\begin{align*}
H^{0,2}(\wt\gg,W)\otimes\bC&\simeq \Lambda^2\tE^* S^2\tH+(\tD+\tE^*)S^3\tH+(\Ad+\Lambda^2\tE^*) S^4\tH+\tE^* S^5\tH\,,\\
H^{1,2}(\wt\gg,W)\otimes\bC&\simeq H^{1,2}(\wt\gg)\otimes\bC\,,
\end{align*}
whereas the cohomology group $H^{2,2}(\wt\gg,W)$ vanishes.
\end{theorem}  

We remark that Theorem \ref{precedente} does not require analiticity assumptions and that $H^{2,2}(\wt\gg,W)=0$ is tantamount to the 
automatic vanishing of the  $(G,\wt G)$-curvature $\cR^3$ of third order. We also note that
$\cR^1$ and $\cR^2$ are well-defined intrinsic global objects on $M$ as a consequence of $H^{1,1}(\wt\gg,W)=H^{2,1}(\wt\gg,W)=0$ (see Proposition \ref{co1ord}). 
If $n=1$ then $\wt G=\GL_{1}(\bH)\cdot\Sp_1$ is isomorphic to $\mathrm{CO}_4^+$ and a four dimensional almost quaternionic manifold is just an (oriented) conformal manifold. We refer the interested reader to \cite{AS} where hypersurfaces of conformal manifolds are considered in the same spirit of Theorem \ref{precedente}.

We make clear that Theorem \ref{precedente} does not say anything about $\cR^1=0$ being sufficient for the existence of a local immersion into a quaternionic manifold; it would be interesting to understand whether or not this is true.
Explicit examples of homogeneous quaternionic contact manifolds are known (see \cite{LC, DC}).
As a simple application of the necessary condition, we give in Theorem \ref{seraven} an explicit construction 
of a family of homogeneous almost CR quaternionic manifolds with a non-trivial essential $(G,\wt G)$-curvature $\cR^1$ at any point and therefore not admitting local immersions in any quaternionic manifold. The construction is based on a deformation of the standard quaternionic contact structure on the quaternionic Heisenberg group.
Since any quaternionic contact manifold always admits an immersion in a quaternionic manifold (see \cite{Duc2}) this family also provides  examples of homogeneous almost CR quaternionic manifolds which are not quaternionic contact.
\vskip0.3cm
\par
The paper is organized as follows. In \S \ref{Preliminaries} we recall the basic definitions of quaternionic linear algebra and geometry; in particular
we describe the structure group $G$ in Lemma \ref{impor} and recall in \S\ref{prelquat}
the basics of $G$-structures $\pi:P\to M$ induced on submanifolds $M$ of manifolds $\wt M$ endowed with an almost quaternionic structure $\wt\pi:\wt P\to\wt M$. In \S \ref{orders} we give a long exact sequence relating the generalized $H^{p,q}(\wt\gg,W)$ and the usual $H^{p,q}(\wt\gg)$ Spencer cohomology groups of $\wt\gg=\ggl_{n}(\bH)\oplus\sp_1$ (Proposition \ref{sux}), calculate the formers for $q=0,1$ (Proposition \ref{co1ord})
and finally prove Theorem \ref{TeoremaRSF} on induced $G$-structures and their essential $(G,\wt G)$-curvatures. 

Sections \S\ref{anticipo}, \S\ref{anticipoanticipo} and \S\ref{anticipoanticipoanticipo} contain the most technical part of the paper. Section \S\ref{anticipo} is devoted to determining the $\gs$-module structure of $H^{0,2}(\wt\gg,W)$. The resulting Theorem \ref{speriamobene} is exploited in \S\ref{esempi!} to construct the  above mentioned family of homogeneous almost CR quaternionic manifolds. In \S\ref{anticipoanticipo} we show that $H^{2,2}(\wt\gg,W)$ vanishes. The proof relies on the  exact sequence of Proposition \ref{sux} and on Proposition \ref{quelloprima}, where the groups $H^{p,3}(\wt\gg)$ are determined with the help of Kostant version of the Borel-Bott-Weil Theorem \cite{Kos}. Finally \S\ref{anticipoanticipoanticipo} provides the canonical isomorphism of $H^{1,2}(\wt\gg,W)$ with $H^{1,2}(\wt\gg)$ and ends with some comments.
\vskip0.3cm
\par
\noindent
{\bf Conventions.}
Let $\gg$ be a Lie algebra. Tensor products of representations of $\gg$ are indicated either in the usual way or simply by juxtaposition.
With $n$ we always indicate a positive integer $n>1$.
\medskip\par\noindent
{\it Acknowledgements.} 
Part of this work was done while the author was a post-doc at the University of Parma. 
The author would like to thank the Mathematics Department and in particular A. Tomassini and 
C. Medori for support and ideal working conditions.
\vskip0.3cm
\section{Preliminaries on quaternionic algebra and geometry}
\label{Preliminaries}
\setcounter{section}{2}
\setcounter{equation}{0}
\subsection{Quaternionic linear algebra}
Let $V$ be a real vector space of dimension $4n$, $n>1$. A hypercomplex structure on $V$ is a triple $H=(I_1,I_2,I_3)$ of anticommuting complex structures on $V$ with $I_3=I_1I_2$. We call the $3$-dimensional subspace $
Q=\operatorname{span}\{I_1,I_2,I_3\}$
of $\End(V)$ a quaternionic structure on $V$ and the pair $(V,Q)$ a quaternionic vector space. We say that $H$ is an {\it admissible basis} of $(V,Q)$ and recall that two bases $H$, $H'$ are always related by an orthogonal matrix $A\in \SO(3)$ as follows (see e.g. \cite{AM}):
\beq
\label{ognialtra}
I'_\alpha=\sum_{\beta=1,2,3}A_\alpha^\beta I_\beta\,.
\eeq
We also say that a basis of $V$ is an {\it admissible frame} for $(V,Q)$ if it is of the form
$$
(I_1 e_1, I_2 e_1, I_3 e_1, e_1, \dots, I_1 e_n, I_2 e_n, I_3 e_n, e_n)\;,
$$
for some admissible basis $H=(I_\alpha)$ and a set $\{e_1,\dots,e_n\}$ of vectors of $V$. 

An isomorphism of quaternionic vector spaces $(V,Q)$ and $(V', Q')$ is an $\bR$-linear invertible map
$
\varphi:V\rightarrow V'
$
satisfying $\varphi^* Q'=Q$ (see e.g. \cite{AM}). We fix once and for all an identification of $(V,Q)$ with 
$\bH^n$  endowed with the quaternionic structure associated to 
$
I_1=-R_i, I_2=-R_j$ and $I_3=-R_k
$,
where $R_h$ is right multiplication by a quaternion $h\in \bH$. 

The following Lemma \ref{impor} describes the subgroup $G_\sharp$ of the automorphism group $\wt G$ of $(V,Q)$ which preserves the subspace $W=\bH^{n-1}\oplus\Im\bH$ of $V=\bH^{n}$. It is convenient to introduce the vector space decomposition
\beq
\label{3dec}
\begin{split}
V&=W\oplus W^\perp\\
&=U\oplus U^\perp \oplus W^\perp\;,
\end{split}
\eeq
with $W^\perp=\left\{(0,\dots,0,h_n)\,|\,h_n\in\bR\right\}$, $U=\bH^{n-1}\subset W$ and $U^\perp=\Im\bH\subset W$. 
\begin{lemma}
\label{impor}
An automorphism $(A,h)\in\wt G$, $A\in\GL_{n}(\bH)$, $h\in\Sp_1$, is in $G_\sharp$ if and only if with respect to the decomposition \eqref{3dec} it is of the form
\begin{equation*}
(A,h)=(\begin{pmatrix} A_1 & A_2 \\ 0 & \lambda h  \end{pmatrix},h)\,,
\end{equation*}
where $A_1\in \GL_{n-1}(\bH)$, $A_2\in\Hom_{\bH}(\bH,\bH^{n-1})$ and $\lambda\in\bR^\times$. In particular the subgroup $N_\sharp$
acting trivially on $W$ is trivial and $G_\sharp$ is naturally identifiable with the quotient group $G=G_\sharp/N_\sharp\subset GL(W)$. 
\end{lemma}
This lemma is proved by straightforward computations whose details are omitted for the sake of brevity.
Note that the Lie algebra $\gg$ of $G$ is isomorphic to the semidirect sum 
$(\ggl_{n-1}(\bH)\oplus\bH)\inplus\Hom_{\bH}(\bH,\bH^{n-1})$
with non-trivial brackets given by the natural ones of $\ggl_{n-1}(\bH)$ and $\bH\simeq \mathfrak{u}_2$ as Lie subalgebras and by $[h,A_2]=-A_2\circ L_h$, 
$[A_1, A_2]=A_1\circ A_2$, where $L_h$ is right multiplication by $h\in\bH$, $A_2\in \Hom_{\bH}(\bH,\bH^{n-1})$ and
$A_1\in \ggl_{n-1}(\bH)$. In particular $\gg$ is not semisimple, with radical given by $(\bR\oplus\bR)\inplus\Hom_{\bH}(\bH,\bH^{n-1})$. From now on we consider
\beq
\label{muccaII}
\gs=\sl_{n-1}(\bH)\oplus\sp_1
\eeq
as fixed Levi factor of $\gg$. We stress that the embedding of $\gs$ in the Lie algebra $\wt\gg=\ggl_{n}(\bH)\oplus\sp_1$ of $\wt G$
is \emph{not} the obvious one but diagonal on the ideal $\sp_1$. 
\subsection{Induced \texorpdfstring{$G$}--structures and CR quaternionic manifolds}
\label{prelquat}
A linear quaternionic structure on a real bundle $E\to \wt M$ of rank $4n$ is a subbundle $Q$ of $\End(E)$ which, around any point $x\in\cU\subset\wt M$, is generated by a field $H=(I_1,I_2,I_3):\cU\to \End(E)$ of hypercomplex structures.
We call $H$ an {\it admissible (local) basis} of $Q$ and remark that two bases on the same $\cU\subset \wt M$ are of the form \eqref{ognialtra} for some $A:\cU\rightarrow SO(3)$. A real manifold $\wt M$ is called an almost quaternionic manifold if $T\wt M$ has a linear quaternionic structure. In this case the associated bundle $\wt\pi:\wt P\rightarrow \wt M$ of admissible frames,
\begin{equation}
\label{eq:eq:eq}
\hskip-1cm\wt P=\left\{\text{frame}\;(I_1 e_1, I_2 e_1, I_3 e_1, e_1,\dots, I_1 e_{n}, I_2 e_n, I_3e_n, e_n)\;\text{of}\;T_x\wt M\;|\right.
$$
$$
\qquad\qquad\qquad\qquad\qquad\hskip1.2cm\left.\phantom{\wt M}H=(I_\alpha)\;\text{admissible basis}\right\},
\end{equation}
is a $\wt G$-structure on $\wt M$ and, if it there is a torsion-free compatible connection, $\wt M$ is called a quaternionic manifold.
 
The quaternionic projective space $\bH$P$^{n}$ is the basic example of quaternionic manifold 
and it is homogeneous $\bH\text{P}^{n}\simeq \PGL_{n+1}(\bH)/P$ ($P=$ stabilizer of quaternionic line)
under the action of an automorphism group of maximal dimension. Indeed, see \cite{AM, Oc2, AS, SS, St}, the Lie algebra $\wt\gg=\ggl_{n}(\bH)\oplus\sp_1$ is of finite type in Cartan sense  and its
{\it maximal transitive prolongation} $\wt\gg_\infty$ is a $\bZ$-graded Lie algebra isomorphic to $\mathfrak{sl}_{n+1}(\bH)$ and with grading 
\beq
\label{maxtrans}
\wt\gg_{\infty}=\wt\gg_{\infty}^{-1}\oplus\wt\gg_{\infty}^{0}\oplus\wt\gg_{\infty}^{1}\,,
\eeq
where $\wt\gg_{\infty}^{-1}=V=\mathbb{H}^n$, $\wt\gg_{\infty}^0=\wt\gg$ and $\wt\gg_{\infty}^{1}\simeq V^*$. 

Any hypersurface $M$ of an almost quaternionic manifold $\wt M$ has a natural structure of an almost CR quaternionic manifold with $E=T\wt M|_M$. In the terminology of \cite{AS} $M$ is a {\it $\wt P$-regular} submanifold of $\wt M$, where $\wt\pi:\wt P\rightarrow \wt M$ is the $\wt G$-structure associated with $\wt M$. This means that at any point $x\in M$ there exists a frame 
$(e_1,\dots,e_{4n})\in \wt P|_x$
with $e_i\in T_xM$ for all $1\leq i\leq 4n-1$. Any such a frame is called {\it adapted} and
$(e_1,\dots,e_{4n-1})$ is the {\it induced frame}.

By the results of \cite{AS} the collection of all induced frames
\begin{equation}
\label{setinduced}
P=\left\{\text{induced frame}\,(e_1,\dots, e_{4n-1})\,\text{of}\; T_xM\,|\, x\in M\right\}
\end{equation}
is a $G$-structure $\pi:P\to M$ on $M$, where the structure group $G\simeq G_\sharp/N_\sharp$. By Lemma \ref{impor} we are allowed to identify $G$ directly with $G_\sharp$. 

Conversely an almost CR quaternionic structure on a manifold $M$ (not necessarily induced by an immersion $M\subset \wt M$) is the same as a $G$-structure $\pi:P\rightarrow M$ with structure group $G=G_\sharp$. Following \cite{AS} we call {\it (locally) immersible} any $\pi:P\rightarrow M$ which is locally of the form \eqref{setinduced} for some immersion in quaternionic projective space $\wt M=\bH$P$^n$.
\par\medskip
\vskip0.3cm
\section{Codimension one integrability conditions}
\label{orders}
\setcounter{section}{3}
\setcounter{equation}{0}
The main aim of this section is to apply the general results of \cite{AS} and describe obstructions to (local) immersibility for an almost CR quaternionic manifold. We first recall the relevant definitions and prove some auxiliary results. 
\subsection{Preliminaries}
Let $\gg$ be the Lie algebra formed by all elements of $\wt\gg=\ggl_n(\bH)\oplus\sp_1$ which preserve the subspace $W=\bH^{n-1}\oplus\Im\bH$ of $V=\bH^n$. Let also $\wt\gg_\infty$ be the maximal transitive prolongation of $\wt\gg$.
\begin{definition}\cite{AS}
\label{RSCG}
The {\it generalized Spencer cohomology groups} $H^{p,q}(\wt\gg, W)$ are the cohomology groups of the differential complex
$$
\dots\overset{\partial}{\longrightarrow} C^{p+1,q-1}(\wt\gg, W)\overset{\partial}{\longrightarrow}C^{p,q}(\wt\gg, W)\overset{\partial}{\longrightarrow} C^{p-1,q+1}(\wt\gg, W)\overset{\partial}{\longrightarrow}\cdots\;,
$$
where $C^{p,q}(\wt\gg, W)=\wt\gg_\infty^{p-1}\otimes\Lambda^q W^{*}$ for all $p,q\geq 0$ 
and  the generalized Spencer operator $\partial$
is given by
$$
\partial c(w_1,\dots,w_{q+1}):=\sum_{i=1}^{q+1}(-1)^{i}[c(w_1,\dots,w_{i-1},\hat{w_i},w_{i+1},\dots,w_{q+1}),w_i]
$$
for all $c\in C^{p,q}(\wt\gg, W)$ and $w_1,\dots,w_{q+1}\in W$.
\end{definition}
There is a relationship between the usual Spencer groups $H^{p,q}(\wt\gg)$ 
of $\wt\gg$ and the generalized ones. We first note  that since each $C^{p,q}(\wt\gg,W)$ carries a natural structure of $\gg$-module for which $\partial$ is equivariant then any group $H^{p,q}(\wt\gg,W)$ has a representation of $\gg$ and in particular of the Levi factor $\gs$. Let $\rho\in V^*$ be a defining 
one-form for $W=\Ker\rho$ and $\res:\Lambda^q V^*\to\Lambda^q W^*$ the restriction map. Let also $\res^{-1}:\Lambda^q W^*\to\Lambda^q V^*$ be the right inverse of $\res$ determined by the vector space decomposition \eqref{3dec}. We consider the short exact sequence $
0\longrightarrow C^{p,q-1}(\wt\gg,W)\longrightarrow C^{p,q}(\wt\gg)\longrightarrow C^{p,q}(\wt\gg,W)\longrightarrow 0
$ of differential complexes induced by the short exact sequence
$$
0\longrightarrow\Lambda^{q-1}W^*\stackrel{\Lambda_\rho \circ \res^{-1}}{\longrightarrow}\Lambda^q V^*\stackrel{\res}{\longrightarrow}\Lambda^q W^*\longrightarrow 0\,,
$$
where $\Lambda_\rho$ is right multiplication by $\rho$. It is easy to see that every morphism in the short exact sequence is equivariant under $\gs$. The associated long exact sequence in cohomology (see e.g. \cite[pag. 17]{BT}) implies the following.
\begin{proposition}
\label{sux}
There exists a long exact sequence of vector spaces
\begin{equation}
\begin{split}
\label{sequenza}
\cdots\longrightarrow H^{p+1,q-1}(\wt\gg, W)\longrightarrow H^{p,q-1}(\wt\gg,W)\longrightarrow H^{p,q}(\wt\gg)\longrightarrow \\
\phantom{ccc}\longrightarrow H^{p,q}(\wt\gg, W)\longrightarrow H^{p-1,q}(\wt\gg,W)\longrightarrow H^{p-1,q+1}(\wt\gg)\longrightarrow\cdots
\end{split}
\end{equation}
which is compatible with the natural action of $\gs$.
\end{proposition}
\medskip\par
We recall that the usual Spencer cohomology groups satisfy (cf. \cite{Sal, Y})
$$H^{0,0}(\wt\gg)\simeq V\;,\;\;\;H^{0,1}(\wt\gg)\simeq \ggl(V)/\wt\gg\;\;\;\;\;\;\text{and}$$
$$H^{p,0}(\wt\gg)=H^{p,1}(\wt\gg)=0\qquad\text{for every}\;\;p\geq 1\,,$$ 
as a direct consequence of definitions and the fact that $\wt\gg_\infty$ is the maximal transitive prolongation of $\wt\gg$. In order to prove our first Theorem \ref{TeoremaRSF} in \S \ref{orders2} we need a similar property for the generalized Spencer groups. 
\begin{proposition}
\label{co1ord}
The groups $H^{p,0}(\wt\gg,W)$ and $H^{p,1}(\wt\gg,W)$ are trivial for all $p\geq 1$ whereas $H^{0,0}(\wt\gg,W)\simeq V$ and $H^{0,1}(\wt\gg,W)\simeq V\otimes W^*/\wt\gg|_W$. 
\end{proposition}
\begin{proof}
It is easy to check directly from definitions that $H^{0,0}(\wt\gg,W)\simeq V$ and $H^{0,1}(\wt\gg,W)\simeq V\otimes W^*/\wt\gg|_W$.
The facts that $H^{p,0}(\wt\gg,W)=0$ for all $p\neq 0, 2$ and $H^{p,1}(\wt\gg,W)=0$ for all $p \geq 3$ follow also immediately from definitions.

The long exact sequence \eqref{sequenza} with $p=2$, $q=0$ gives
$$
0=H^{2,0}(\wt\gg)\longrightarrow H^{2,0}(\wt\gg,W)\longrightarrow H^{1,0}(\wt\gg,W)=0
$$
so that $H^{2,0}(\wt\gg,W)=0$. Similarly if $p=3$, $q=1$ we have
$$
0=H^{3,1}(\wt\gg,W)\longrightarrow H^{2,1}(\wt\gg,W)\longrightarrow H^{2,2}(\wt\gg)
$$
and $H^{2,1}(\wt\gg,W)=0$ too, as $H^{2,2}(\wt\gg)=0$ by \cite[Theorem 3.4]{Sal}. 

We are left with the vanishing of $H^{1,1}(\wt\gg,W)$. The proof of this fact is more involved and relies on some of the results of \S \ref{anticipo}, that is Theorem \ref{speriamobene} and Proposition \ref{primedecomp}, which give (after complexification) the following decompositions  into irreducible and inequivalent $\gs$-modules:
\begin{equation}
\label{eq:eq:eq1}
\begin{split}
H^{0,2}(\wt\gg,W)&\simeq \Lambda^2\tE^* S^2\tH+(\tD+\tE^*)S^3\tH\\
&\;\;\;\;+(\Ad+\Lambda^2\tE^*) S^4\tH+\tE^* S^5\tH\,,\\
H^{0,2}(\wt\gg)&\simeq  \tE^* \tH+
(\bC+\Ad+\Lambda^2 \tE^*) S^2\tH+
(2\tE^*+\tE+\tD) S^3\tH\\
&\;\;\;\;+ (\bC+\Ad+\Lambda^2\tE^*) S^4\tH+
\tE^* S^5\tH\,.
\end{split}
\end{equation}
Consider now the exact sequence
\begin{equation*}
0\longrightarrow H^{1,1}(\wt\gg,W)\longrightarrow H^{0,1}(\wt\gg,W)\stackrel{\varphi}{\longrightarrow} H^{0,2}(\wt\gg)\longrightarrow H^{0,2}(\wt\gg,W)\longrightarrow 0
\end{equation*}
given by \eqref{sequenza} with $p=1$, $q=1$. Since any representation of $\gs$ is completely reducible, equation
\eqref{eq:eq:eq1} together with $H^{0,2}(\wt\gg,W)\simeq H^{0,2}(\wt\gg)/\Im\varphi$ implies
\beq
\label{uf1}
\Im\varphi\simeq \tE^*\tH+(\bC+\Ad)S^2\tH+(\tE^*+\tE) S^3\tH+S^4\tH\ .
\eeq
However $H^{0,1}(\wt\gg,W)\simeq V\otimes W^*/\wt\gg|_W$ is also isomorphic with the right hand side of \eqref{uf1} (use that 
$\wt\gg\simeq \wt\gg_{|W}$) so that $\varphi$ is injective and $H^{1,1}(\wt\gg,W)=0$.
\end{proof}
\subsection{Induced \texorpdfstring{$G$}--structures and generalized integrability conditions}
\label{orders2}
Let $M$ be an almost CR quaternionic manifold and $\pi:P\to M$ its associated $G$-structure. 
Let also $\vartheta=(\vartheta^1,\ldots,\vartheta^{4n-1}):TP\rightarrow W$ be the so-called {\it soldering form} of $P$, defined by
$
\vartheta_u(v)=(v^1,\dots,v^{4n-1})
$,
where the $v^i$ are the components of the vector $\pi_*(v)\in T_{\pi(u)}M$ with respect to the frame $u=(e_i)$.
From the general results of \cite{AS} $P$ is locally immersible if and only if, for any local section $s:\cU\subset M\rightarrow P$ of $P$, there exist $1$-forms 
$$\omega^{p}:T\cU\rightarrow \wt\gg_\infty^{p}\;,\qquad -1\leq p\leq 2\,,$$ 
with
\beq
\label{MCC}
\!\!\omega^{\mathrm{-1}}=(s^*\vartheta^1,\ldots,s^*\vartheta^{4n-1},0)
\eeq
and the others satisfying the system of $3$ equations
\beq
\begin{split}
\label{gradiamoII}
&d\omega^{p-1}+\frac{1}{2}\sum_{r=0}^{p-1}[\omega^r,\omega^{p-1-r}]=-[\omega^{-1},\omega^p]\;,\;\; 0\leq p\leq 1\,,\\
&d\omega^{1}+\frac{1}{2}\sum_{r=0}^{1}[\omega^r,\omega^{1-r}]=0\ .
\end{split}
\eeq
The $G$-structure $P$ is called {\it locally immersibile up to order $p$} if, around any point, it admits a local section $s$ and forms $\omega^{-1},\dots,\omega^{p-1}$ which satisfy \eqref{MCC} and the first $p$ equations of \eqref{gradiamoII}. We stress the fact that, if \eqref{MCC} and the first $p$ equations of \eqref{gradiamoII} hold for a choice of $s$, then they hold for any other section of $P$ on the same open set $\cU$.

Any $p+1$-tuple $(s, \o^0, \ldots, \o^{p-1})$ as above is called {\it admissible} and the corresponding $\wt\gg_\infty^{p-1}$-valued $2$-form on the domain of $s$ given by
$$
\Omega^{p-1}=d\omega^{p-1}+\frac{1}{2}\sum_{r=0}^{p-1}[\omega^r,\omega^{p-1-r}]
$$
is the {\it total $(G, \wt G)$-curvature of order $p+1$ associated with $(s, \o^0, \ldots, \o^{p-1})$}. 

Assume $P$ is locally immersible up to order $p$ and, for any admissible $p+1$-tuple $(s,\omega^0,\ldots,\omega^{p-1})$ on some open set $\cU\subset M$, use the frames $s_x$, $x\in\cU$, to identify the tangent spaces $T_xM$ with $W$. By \cite[Theorem 3.4]{AS} the total $(G,\wt G)$-curvature tensors
$$\Omega^{p-1}|_x\in \wt\gg_\infty^{p-1}\otimes\Lambda^{2}T^*_xM\simeq C^{p,2}(\wt\gg,W)$$
satisfy $\partial (\Omega^{p-1}|_x)=0$. The quotient map
\beq
\label{leon}
\cR^{p+1}:\cU\rightarrow H^{p,2}(\wt\gg,W)\;,\;\;\;\;\;\cR^{p+1}_x:=[\Omega^{p-1}|_x]
\eeq
is called {\it essential $(G,\wt G)$-curvature of order $p+1$}. 

One can easily see from definitions that the essential $(G,\wt G)$-curvature $\cR^1$ is the cohomology class in $H^{0,2}(\wt\gg,W)$ of the torsion tensor associated with any linear connection on $\pi:P\to M$. One can therefore consider its
 vanishing as a generalization to the context of induced $G$-structures of the usual vanishing of the intrinsic torsion.
It also admits the following simple geometric interpretation: if the almost CR quaternionic structure is induced by a local immersion in a quaternionic manifold $\wt M$ then $\cR^1=0$. 

This can be seen as follows. Consider a local section $\wt s:\wt \cU\to \wt P$ of the $\wt G$-structure $\wt \pi:\wt P\to \wt M$ associated with $\wt M$. Without loss of generality we may assume that $$\wt s|_{\cU}:\cU\to \wt P\;,\qquad \cU=\wt\cU\cap M\,,$$ is a field of adapted frames with field of induced frames $s:\cU\to P$. The co-frames associated with $\wt s$ and $s$ are given by the local $1$-forms
$$\wt\omega^{-1}:T\wt\cU\longrightarrow V\qquad\text{and}\qquad \omega^{-1}=\wt\omega^{-1}|_{T\cU}:T\cU\longrightarrow W\ .$$
Since $\wt M$ is quaternionic there exists a torsion-free connection on $\wt P$. The corresponding connection $1$-form $\wt\omega^0:T\wt\cU\to \wt\gg$ satisfies
$
d\wt\omega^{-1}=-[\wt\omega^{-1},\wt\omega^0]
$. It follows that $d\omega^{-1}=-[\omega^{-1},\omega^0]
$ where $\omega^0=\wt\omega^0|_{T\cU}$ and hence $\cR^1=0$.

We now state the main result of \S \ref{orders2}. It is a direct consequence of above observations, Proposition \ref{co1ord}, the fact
that the $\wt G$-structure $\wt\pi:\wt P\longrightarrow\bH${\rm P}$^n$ is {\it flat $k$-reductive} with $k=2$ in the sense of \cite{AS} and of
\cite[Theorem 4.3]{AS}.
\begin{theorem}
\label{TeoremaRSF}
Let $M$ be an almost CR quaternionic manifold of dimension $4 n-1\geq 7$ and $\pi:P\rightarrow M$ its associated $G$-structure.
Then there exists a $G$-equivariant map $\cR^{1}:P\rightarrow H^{0,2}(\wt\gg, W)$ which vanishes if and only if $P$ is locally immersible up to first order. In particular if $P$ is induced by a local immersion $M\subset \wt M$ into a quaternionic manifold $\wt M$
then $\cR^1=0$. 

Assume $\cR^1=0$. Then there is a $G$-equivariant $\cR^{2}:P\rightarrow H^{1,2}(\wt\gg,W)$
which vanishes if and only if $P$ is locally immersible up to second order. Finally if 
$\cR^1=\cR^2=0$ there is a $G$-equivariant $\cR^{3}:P\rightarrow H^{2,2}(\wt\gg,W)$
and $P$ is locally immersible into $\wt\pi:\wt P\longrightarrow\bH${\rm P}$^n$ if and only if 
$\cR^1=\cR^2=\cR^3=0$. 
\end{theorem}
We observe that Theorem \ref{TeoremaRSF} says that the essential $(G,\wt G)$-curvatures are intrinsically defined for any almost CR quaternionic manifold. This is a non-trivial result, which depends on the vanishing of $H^{p,1}(\wt\gg,W)$ for all $p\geq 1$. For instance, in the classical case of Riemannian immersions, $G=\SO_{n}(\bR)$, $\wt G=\SO_{\wt n}(\bR)$ and
$H^{1,1}(\so_{\wt n}(\bR), \bR^n)\simeq\bR^{\wt n-n}\otimes S^2 (\bR^n)^*+\so(\bR^{\wt n-n})\otimes (\bR^n)^*$ \cite{AS}. This reflects the fact that the generalized $(G,\wt G)$-curvatures whose vanishing is equivalent to the Gauss-Codazzi-Ricci equations depend on the choice of a candidate second fundamental form and normal metric connection.

The main aim of next sections is to improve Theorem \ref{TeoremaRSF} with the explicit description of the groups $H^{p,2}(\wt\gg,W)$. In particular
$H^{2,2}(\wt\gg,W)=0$ so that the third obstruction $\cR^3$ is automatically trivial. From now on we will work exclusively over $\bC$ so that real vector spaces, Lie algebras and representations are tacitly complexified. We collect our conventions in the following.
\subsection{Conventions on complexifications}
\label{sec:conv}
We recall that there exists an isomorphism $
V\simeq \wt\tE\otimes \tH$ of $\wt \gg$-modules, $\wt \gg=\ggl_{2n}(\bC)\oplus\sl_{2}(\bC)$,
where $\wt\tE$ and $\tH$ are the defining representations of $\ggl_{2n}(\bC)$ and $\sl_{2}(\bC)$. 
If we consider $\wt\tE\simeq\bH^n$ and $\tH\simeq\bH$ with the complex structures $R_i:\wt\tE\rightarrow \wt\tE$ and $L_i:\tH\rightarrow\tH$, the underlying real representation of $\ggl_{n}(\bH)\oplus\sp_1$ is recovered as the fixed points set of the conjugation
$
\tau=R_j\otimes L_j:\wt\tE\otimes\tH\longrightarrow \wt\tE\otimes\tH
$. We also fix an $\sl_2(\bC)$-invariant symplectic form $\omega$ on $\tH$ satisfying
$\omega(L_jr,L_js)=-\overline{\omega(r,s)}$.

The Levi subalgebra \eqref{muccaII} of $\gg$ is given by
\begin{equation*}
\gs\simeq\sl_{2n-2}(\bC)\oplus\sl_{2}(\bC)\longrightarrow \wt\gg\;,
\end{equation*}
with the ``diagonal'' embedding of $\sl_2(\bC)$ in $\wt\gg$. Its action on $V$ is compatible with the decomposition \eqref{3dec}: there exists a decomposition $
\wt\tE=\tE+\tH$ of $\gs$-modules, where $\tE$ is the defining representation of $\sl_{2n-2}(\bC)$, and
\begin{equation}
\label{serve}
V\simeq \wt\tE\tH\simeq \tE\tH + \tH\tH\simeq \tE\tH+ S^2\tH +\bC\,,\qquad W\simeq \tE\tH+ S^2\tH\;.
\end{equation}

For our purposes it is convenient to read these facts inside the maximal transitive prolongation $\wt\gg_\infty\simeq \sl_{2n+2}(\bC)$ of $\wt\gg$. 
Let $\gt$ be the Cartan subalgebra of trace-free diagonal matrices of $\sl_{2n+2}(\bC)$, $\Phi$ the corresponding set of roots and
$\gg_\delta$ the root space of $\delta\in\Phi$. We fix a simple root system $\{\delta_1,\dots,\delta_{\ell}\}$ of $\Phi$, $\ell=2n+1$, 
and let $H_{\delta_i}\in\gt$ be the coroot associated to the simple root $\delta_i$. 

If we set $\operatorname{deg}(\sum_{i=1}^{\ell} c_i\delta_i)=c_2$ to be the degree of a root, the grading \eqref{maxtrans} corresponds to the decomposition
$\sl_{2n+2}(\bC)=\wt\tE\tH+\wt\gg+\wt\tE^*\tH$ where
\vskip0.1cm\par\noindent
\begin{align}
\label{eq:gradingprol}
\begin{cases}
\wt\gg=\mathfrak{t}+
\displaystyle\sum_{\substack{\delta\in\Phi\\\deg(\delta)=0}}\gg_\delta,\\
\wt\tE\tH=\;\!\!\!\displaystyle\sum_{\substack{\delta\in\Phi\\\deg(\delta)=-1}}\gg_\delta,\\
\wt\tE^*\tH=\,\displaystyle\sum_{\substack{\delta\in\Phi\\\deg(\delta)=+1}}\gg_\delta.
\end{cases}
\end{align}
\phantom{$$
\begin{tikzpicture}
\node[root]   (1)                     {};
\node[xroot] (2) [right=of 1] {} edge [-] (1);
\node[]   (3) [right=of 2] {$\;\cdots\,$} edge [-] (2);
\node[root]   (4) [right=of 3] {} edge [-] (3);
\node[root]   (5) [right=of 4] {} edge [-] (4);
\end{tikzpicture}
$$}
\vskip-0.3cm\par\noindent
In particular the simple ideals of $\wt\gg$ are given by 
\begin{align}
\label{ebbenesi}
\sl_2(\bC)&=\left\langle H_{\delta_1}, E_{\delta_1}, E_{-\delta_1}\right\rangle\\
\label{ebbeneno}
\sl_{2n}(\bC)&=\left\langle H_{\delta_3}, E_{\delta_3}, E_{-\delta_3},\dots,H_{\delta_{\ell}}, E_{\delta_{\ell}}, E_{-\delta_{\ell}}\right\rangle
\end{align}
with Cartan subalgebras generated  by $\{H_{\delta_1}\}$ and $\{H_{\delta_3},\dots,H_{\delta_{\ell}}\}$, respectively. 
On the other hand $\gs$  is the direct sum of a copy 
\begin{equation*}
\sl_2(\bC)=\left\langle H_{\delta_1}+H_{\delta_\ell}, E_{\delta_1}+E_{\delta_\ell},  E_{-\delta_1}+E_{-\delta_\ell}\right\rangle        
\end{equation*}
different from \eqref{ebbenesi} and the following Lie subalgebra of \eqref{ebbeneno}
\begin{equation*}
\sl_{2n-2}(\bC)=\left\langle H_{\delta_3}, E_{\delta_3}, E_{-\delta_3},\dots, H_{\delta_{\ell-2}}, E_{\delta_{\ell-2}}, E_{-\delta_{\ell-2}}\right\rangle\ .
\end{equation*}
Finally the decomposition \eqref{serve} is realized by
$$
\tE\tH=\left\langle E_{\delta}\,|\,\delta=c_1\delta_1+\cdots+c_{\ell}\delta_{\ell}\in\Phi^-\;\;\text{with}\;\; c_2=-1,\, c_{\ell-1}=c_{\ell}=0\right\rangle\,,\phantom{cc}
$$
$$
S^2\tH=\left\langle E_{-(\delta_1+\cdots+\delta_{\ell-1})}-E_{-(\delta_2+\cdots+\delta_\ell)}, E_{-(\delta_2+\cdots+\delta_{\ell-1})},\, E_{-(\delta_1+\cdots+\delta_\ell)} \right\rangle\,,\phantom{cccc}
$$
$$
\bC=\left\langle E_{-(\delta_1+\cdots+\delta_{\ell-1})}+E_{-(\delta_2+\cdots+\delta_\ell)}\right\rangle\,.\phantom{ccccccccccccccccccccccccccccci}
$$
We will use these facts several times in \S \ref{anticipo}, \S \ref{anticipoanticipo} and \S\ref{anticipoanticipoanticipo}.
We will also extensively use that the irreducible representations of $\sl_2(\bC)$ are (up to isomorphism) given by the symmetric powers $S^k\tH$ and that tensor products behave accordingly to the Clebsch-Gordan formula
$
S^j\tH\otimes S^k\tH\simeq \bigotimes_{r=0}^{\text{min}(j,k)}S^{j+k-2r}\tH\ .
$
On the other hand we will identify irreducible representations of (higher rank) semisimple Lie algebras with weights or Young diagrams. 
In particular any irreducible representation of $\sl_{2n}(\bC)\oplus\sl_2(\bC)$ is of the form $M\otimes S^k\tH$ for some irreducible representation $M$ of $\sl_{2n}(\bC)$ and a nonnegative integer $k$. If $M\subset \otimes^p\wt\tE\otimes^q\wt\tE^*$ with $p+q+k$ even, it is (the complexification of) a real representation of $\ggl_{n}(\bH)\oplus\sp_1$. 
\vskip0.2cm\par
\section{The essential \texorpdfstring{$(G,\wt G)$}--curvature \texorpdfstring{$\cR^1:P\rightarrow H^{0,2}(\wt\gg,W)$}{}}
\setcounter{section}{4}
\setcounter{equation}{0}
\label{anticipo} 
\vskip0.2cm\par
The main goal of this section is to describe the group $H^{0,2}(\wt\gg,W)$ and consider an application to the construction of homogeneous almost CR quaternionic manifolds. More precisely we have the following two results.
\begin{theorem}
\label{speriamobene}
The cohomology group $H^{0,2}(\wt\gg,W)$ decomposes into the direct sum 
$
H^{0,2}(\wt\gg,W)\simeq\Lambda^2\tE^* S^2\tH+(\tD+\tE^*)S^3\tH+(\Ad+\Lambda^2\tE^*) S^4\tH+\tE^* S^5\tH
$
of six irreducible and inequivalent $\gs$-modules.
\end{theorem}
\begin{theorem}
\label{thm:deformazioni}
There exists a $1$-parameter family of almost CR quaternionic structures $\pi:P_t\to M$, $t\geq 0$, on the quaternionic Heisenberg group with $P_t$ isomorphic to the standard quaternionic contact structure only at $t=0$.
\end{theorem}
\subsection{The main result}
Our strategy to prove Theorem \ref{speriamobene} is the following. 
First the space
of generalized $(0,2)$-cocycles coincides with $V\otimes\Lambda^2 W^*$
so that $H^{0,2}(\wt\gg,W)=V\otimes\Lambda^2 W^*/B^{0,2}(\wt\gg,W)$. Consider then
the restriction operator $\res:V\otimes\Lambda^2V^*\to V\otimes \Lambda^2 W^*$;
one easily checks from definitions that it is surjective, $\gs$-equivariant and satisfies
$
B^{0,2}(\wt\gg,W)=\res(B^{0,2}(\wt\gg))$. It follows that there exist appropriate irreducible $\gs$-modules contained in $B^{0,2}(\wt\gg)$ which are faithfully preserved by $\res$. We will decompose $B^{0,2}(\wt\gg)$ into irreducible $\gs$-submodules, determine $B^{0,2}(\wt\gg,W)$ and, in turn, the group $H^{0,2}(\wt\gg,W)$.

The $\wt\gg$-module structure of $H^{0,2}(\wt\gg)$ is well-known. If $\wt\tC$ and $\wt\tD$ are the kernels of the natural contractions $\wt\tE\otimes S^2\wt\tE^*\rightarrow \wt\tE^*$ and $\wt\tE\otimes \Lambda^2\wt\tE^*\rightarrow \wt\tE^*$ then:
\begin{align}
V\otimes \Lambda^2 V^*&\simeq \wt\tE \tH\otimes \Lambda^2(\wt\tE^*\tH)\simeq\wt\tE\tH\otimes (S^2 \wt\tE^*\Lambda^2\tH+\Lambda^2\wt\tE^* S^2\tH)\notag\\
&\simeq (\wt\tE S^2\wt\tE^*\otimes \tH)+(\wt\tE\Lambda^2 \wt\tE^*\otimes \tH S^2\tH)\notag\\
&\simeq (\wt\tE^*+ \wt\tC)\tH+(\wt\tE^*+ \wt\tD)\tH S^2\tH\label{falsi}\\
&\simeq 2\wt\tE^*\tH+\wt\tC\tH+\wt\tD\tH+ \wt\tE^* S^3\tH+ \wt\tD S^3\tH\notag\\
B^{0,2}(\wt\gg)&\simeq 2\wt\tE^*\tH+\wt\tC\tH+\wt\tD\tH+ \wt\tE^* S^3\tH\notag
\end{align}
so that $H^{0,2}(\wt\gg)\simeq\wt\tD S^3\tH$ (cf. \cite[Proposition 2.2]{Sal}). We now collect a series of intermediate useful results.
\begin{proposition}
\label{primedecomp}
The  decompositions into irreducible and inequivalent $\gs$-modules of $\Lambda^2 W^*$, $W\otimes \Lambda^2 W^*$, $V\otimes \Lambda^2 W^*$,  $B^{0,2}(\wt\gg)$, 
the kernels $\wt\tC$ and $\wt\tD$ of the natural contractions $\wt\tE\otimes S^2\wt\tE^*\rightarrow \wt\tE^*$, $\wt\tE\otimes \Lambda^2\wt\tE^*\rightarrow \wt\tE^*$ and $\wt\tD S^3\tH$ are: 
\vskip0.3cm\par
\begin{align*}
\Lambda^2 W^*\simeq &S^2\tE^*+\tE^*\tH+(\Lambda^2\tE^*+\bC)S^2\tH+\tE^* S^3\tH\;,\\
W\otimes \Lambda^2 W^*\simeq & (2\bC+\Ad+\Lambda^2 \tE^*)+(\tE+\tD+\tC+4\tE^*)\tH+\\&
(3\bC+2\Ad+\Lambda^2 \tE^*+ S^2\tE^*) S^2 \tH+(\tE+\tD+3\tE^*)S^3\tH+\\
&(2\bC+\Ad+\Lambda^2\tE^*)S^4\tH+\tE^* S^5\tH\;,\\
V\otimes \Lambda^2 W^*\simeq & (2\bC+\Ad+\Lambda^2 \tE^*+S^2\tE^*)+(\tE+\tD+\tC+5\tE^*)\tH+\\&
(4\bC+2\Ad+2\Lambda^2 \tE^*+ S^2\tE^*) S^2 \tH+(\tE+\tD+4\tE^*)S^3\tH+\\
&(2\bC+\Ad+\Lambda^2\tE^*)S^4\tH+\tE^* S^5\tH\;,\\
\wt\tC\simeq &(\tE^*+\tC)+(S^2\tE^*+\bC+\Ad)\tH+(\tE^*+\tE)S^2\tH +S^3\tH\;,\\
\wt\tD\simeq  &(\tE^*+\tE+\tD)+
(\Lambda^2\tE^*+\bC+\Ad)\tH+
\tE^* S^2\tH\;,\\
B^{0,2}(\wt\gg)\simeq & (S^2 \tE^*+4\bC+2\Ad+\Lambda^2 \tE^*)+(6\tE^*+\tC+2\tE+\tD)\tH+\\
& (S^2\tE^*+6\bC+2\Ad+\Lambda^2\tE^*)S^2\tH+(3\tE^*+\tE)S^3\tH+2 S^4\tH\;,\\
\wt\tD S^3\tH\simeq & \tE^* \tH+(\bC+\Ad+\Lambda^2 \tE^*) S^2\tH+(2\tE^*+\tE+\tD) S^3\tH+\\
&(\bC+\Ad+\Lambda^2\tE^*) S^4\tH+\tE^* S^5\tH\;.
\end{align*}
\end{proposition}
\begin{proof}
We recall that $W^*\simeq \tE^*\tH+ S^2\tH$. The first decomposition is a consequence of $\Lambda^2 W^*\simeq \Lambda^2(\tE^*\tH)+\Lambda^2(S^2\tH)+\tE^*\tH S^2\tH$, $\Lambda^2 (S^2\tH)\simeq S^2\tH$ and it directly implies
the decomposition of $W\otimes \Lambda^2 W^*$. 

The third decomposition follows from $V\otimes\Lambda^2 W^*\simeq\Lambda^2 W^*+W\otimes\Lambda^2 W^*$ while
that of $\wt\tC\simeq (\wt\tE S^2\wt\tE^*)/ \wt\tE^*$ from identity
$$
\wt\tE S^2\wt\tE^*\simeq (2\tE^*+\tC)+(S^2\tE^*+2\bC+\Ad)\tH+(\tE^*+\tE)S^2\tH+S^3\tH\;.
$$
The proof for $\wt\tD$ is similar. Finally the decomposition of $B^{0,2}(\wt\gg)$ follows from those of $\wt\tC$, $\wt\tD$:
\begin{align*}
B^{0,2}(\wt\gg)& \simeq 2\wt\tE^*\tH+\wt\tC\tH+\wt\tD\tH+\wt\tE^* S^3\tH\\ & \simeq 2\bC+(2\tE^*+\wt\tC+\wt\tD)\tH+3S^2\tH+\tE^* S^3\tH+S^4\tH
\end{align*}
and that of $\wt\tD S^3\tH$ directly from the decomposition of $\wt\tD$.
\end{proof}
\begin{proof}[Proof of Theorem \ref{speriamobene}]
The proof split into several steps.

(i) Since $\res:V\otimes\Lambda^2V^*$ to $V\otimes \Lambda^2 W^*$ is $\gs$-equivariant and surjective, standard properties of representation theory of semisimple Lie algebras and Proposition \ref{primedecomp} immediately yield
\begin{equation}
\label{eq:firstsummary}
\begin{split}
B^{0,2}(\wt\gg,W)\supset &(2\bC+\Ad+\Lambda^2 \tE^*+S^2\tE^*)+(\tE+\tD+\tC)\tH+\phantom{cccccccc}\\
&(\Lambda^2\tE^*+S^2\tE^*)S^2 \tH\;,\\
H^{0,2}(\wt\gg,W)\supset& \Lambda^2\tE^* S^2\tH+\tD S^3\tH+
(\Ad+\Lambda^2\tE^*) S^4\tH+\tE^* S^5\tH\ .
\end{split}
\end{equation}
We however require  a deeper analysis to determine the extent to which the remaining $\gs$-submodules of $V\otimes\Lambda^2 W^{*}$,
$$5\tE^*\tH\;,\;\; (4\bC+2\Ad) S^2 \tH\;,\;\; (\tE+4\tE^*) S^3\tH\;,\;\; 2 S^4\tH\;,$$
either belong to $B^{0,2}(\wt\gg,W)$ or contribute nontrivially to $H^{0,2}(\wt\gg,W)$. A second look at Proposition \ref{primedecomp} tells us that at least one $\tE^* S^3\tH$ has to contribute to the cohomology. We will actually see that $B^{0,2}(\wt\gg,W)$ contains the full isotipic component in $V\otimes\Lambda^2 W^{*}$ of $\tE^*\tH$, $S^2\tH$, $\Ad S^2 \tH$, $\tE S^3\tH$, $S^4\tH$ but three copies of $\tE^* S^3\tH$.

In order to specify $\gs$-equivariant immersions, we consider the projectors 
$\xi:\wt\tE\rightarrow \tE$ and $\psi:\wt\tE\rightarrow \tH$
associated with the decomposition $\wt\tE=\tE+\tH$. We define immersions of $\wt\tE^*$ into $\wt\tE\otimes S^2\wt\tE^*$ and $\wt\tE\otimes \Lambda^2\wt\tE^*$ by
letting $f\in \wt\tE^*$ act on $x,y\in\wt\tE$ by $f(x)y\pm f(y)x$ and, in a similar way, also immersions of $\tE^*$ in $\tE\otimes S^2\tE^*$ and $\tE\otimes \Lambda^2\tE^*$. Finally we describe the elements of $\wt\tD\subset \wt\tE\otimes\Lambda^2 \wt\tE^*$ directly by their action on decomposable tensors and consider the immersions $\beta_1:\tE^*\rightarrow \wt\tD$, $\beta_2:\tH\rightarrow \wt\tD$ and $\beta_3:\tE^* S^2\tH\rightarrow \wt\tD$ defined by
\begin{align*}
\beta_1(e)&=e(x)(\xi-\frac{2n-3}{2}\psi)y-e(y)(\xi-\frac{2n-3}{2}\psi)x\,,\\
\beta_2(h)&=h(x)((2n-2)\psi-\xi)y-h(y)((2n-2)\psi-\xi)x\,,\\
\beta_3(e\odot^2h)&=[e(x)h(y)-e(y)h(x)] h\,,
\end{align*}
where $e\in\tE^*$, $h\in\tH$.

(ii) Consider the $\wt\gg$-submodule $\wt\tE^*\tH+\wt\tC\tH\simeq\wt\tE\tH\otimes S^2\wt\tE^*\Lambda^2\tH\subset V\otimes \Lambda^2 V^*$. We see from \eqref{falsi}   that it is contained in $B^{0,2}(\wt\gg)$  and
\begin{align*}
\wt\tE\tH\otimes S^2\wt\tE^*\Lambda^2\tH\simeq & (2\bC+\Ad+S^2 \tE^*)+(\tE+\tC+3\tE^*)\tH +\\
& (3\bC+\Ad+S^2\tE^*)S^2 \tH+(\tE+\tE^*)S^3\tH + S^4\tH\ .\\
\end{align*}
\vskip-0.4cm\noindent
As advertised before, we are interested only in the $\gs$-isotipic components of $\tE^*\tH$, $S^2\tH$, $\Ad S^2 \tH$, $\tE S^3\tH$, $\tE^* S^3\tH$ and $S^4\tH$. We again describe elements in the images of the immersions
\begin{align*}
\imath_1:&\tE^*\tH\rightarrow \tE\tH\otimes S^2\tE^*\Lambda^2\tH\;,\;
&\imath_2&:\tE S^3\tH\rightarrow \tE\tH\otimes S^2\tH\Lambda^2\tH\;\phantom{cccccccccccccccccccccccccccccccccccccccc}\\
\imath_3:&S^2\tH\rightarrow \tE\tH\otimes (\tE^*\odot\tH)\Lambda^2\tH\;,\;
&\imath_4&:\Ad S^2\tH\rightarrow \tE\tH\otimes (\tE^*\odot\tH)\Lambda^2\tH\;\\
\imath_5:&S^4\tH\rightarrow S^2\tH \otimes S^2\tH\Lambda^2\tH\;,\;
&\imath_6&:S^2\tH\rightarrow S^2\tH \otimes S^2\tH\Lambda^2\tH\;\\
\imath_7:&\tE^* S^3\tH\rightarrow S^2\tH \otimes (\tE^*\odot\tH)\Lambda^2\tH\;,\;\hskip-0.4cm
&\imath_8&:\tE^*\tH\rightarrow S^2\tH \otimes (\tE^*\odot\tH)\Lambda^2\tH\;\\
\imath_9:&\tE^*\tH\rightarrow \bC \otimes (\tE^*\odot \tH)\Lambda^2\tH\;,\;
&\imath_{10}&: S^2\tH\rightarrow \bC\otimes S^2\tH \Lambda^2\tH\;\\
\end{align*}
\vskip-0.5cm\par\noindent
directly by their action on decomposable tensors:
\begin{align*}
&\imath_1(e\otimes h)=(e(x)\xi y+e(y)\xi x)\otimes h \omega(r,s)\;,\\
&\imath_2(e^*\odot^3h)= e^*h(x)h(y)\otimes h\omega(r,s)\;,\\
&\imath_3(\odot^2 h)=(h(x)\xi y+h(y)\xi x) \otimes h \omega(r,s)\;,\\
&\imath_4(A\odot^2 h)=(h(x)A(\xi y)+h(y)A(\xi x)) \otimes h \omega(r,s)\;,\\
&\imath_5(\odot^4 h)=h(x)h(y)h \otimes h\, \omega(r,s)\;,\\
&\imath_6(\odot^2 h)=\ad_{\odot^2 h}(\psi x \odot\psi y)\omega(r,s)\;,\\
&\imath_7(e\odot^3 h)=(e(x)h(y)+e(y) h(x)) h \otimes h\, \omega(r,s)\;,\\
&\imath_8(e\otimes h)=[e(x)(\psi y\otimes h-h(y)/2)+e(y)(\psi x\otimes h-h(x)/2)] \omega(r,s)\;,\\
&\imath_9(e\otimes h)=[e(x)h(y)+e(y) h(x)]\, \omega(r,s)\; , \\
&\imath_{10}(\odot^2 h)=h(x)h(y)\, \omega(r,s)\; ,
\end{align*}
where $x,y\in\wt\tE$ and $r,s\in \tH$. Recall the decomposition \eqref{3dec} (see also \eqref{serve}). The following Table \ref{table:action} summarizes
the relevant information on the elements of $\Im(\imath_i)$, for each $i=1,\ldots,10$. The row ``domain'' displays the subspace of $\Lambda^2 V$ where the (non-zero) elements of $\Im(\imath_i)$ acts non trivially, the row ``codomain'' the subspace of $V$ where elements take their values. 
\begin{table}[H]
\begin{centering}
\makebox[\textwidth]{
\begin{tabular}{|c|c|c|c|c|c|}
\hline
$i$ & $1$ & $2$ & $3$ & $4$ & $5$\\
\hline
&&&&&\\[-3mm]
\text{domain} & $\Lambda^2 U$ & $\Lambda^2 U^\perp$ & $U\wedge U^\perp$ & $U\wedge U^\perp$ & $\Lambda^2 U^\perp$
\\
\hline
\text{codomain} & $U$ & $U$ & $U$ & $U$ & $U^\perp$
\\
\hline
\hline
$i$ & $6$ & $7$ & $8$ & $9$ & $10$ \\
\hline
&&&&&\\[-3mm]
\text{domain} & $\Lambda^2 U^\perp$  & $U\wedge U^\perp$ & $U\wedge U^\perp$ & $U\wedge U^\perp$ & $\Lambda^2 U^\perp$
\\
\hline
\text{codomain} & $U^\perp$  & $U^\perp$ & $U^\perp$ & $W^\perp$ & $W^\perp$
\\
\hline
\end{tabular}}
\end{centering}
\caption[]{Elements of $\Im(\imath_i)\subset V\otimes\Lambda^2 V^*$.\label{table:action}} \vskip14pt
\end{table}
\vskip-0.7cm\par\noindent
It follows that 
\begin{align*}
B^{0,2}(\wt\gg,W)&\supset \bigoplus_{i=1}^{10} \res(\Im(\imath_i))\\
&\simeq 3\tE^*\tH+(3\bC+\Ad)S^2\tH+(\tE+\tE^*)S^3\tH+S^4\tH\;.
\end{align*}
We now perform a similar analysis for the other $\wt\gg$-submodules of $V\otimes\Lambda^2 V^*$; we will not give all the details as in step (ii) but just the main points.

(iii) Consider $\wt\tE^*\tH+\wt\tD\tH+\wt\tE^* S^3\tH\subset \wt\tE\tH\otimes \Lambda^2\wt\tE^* S^2\tH^*\subset V\otimes\Lambda^2 V^*$. 
It is contained in $B^{0,2}(\wt\gg)$ by \eqref{falsi} and decomposes into $\gs$-modules as follows:
\begin{align*}
\wt\tE^*\tH+\wt\tD\tH+\wt\tE^* S^3\tH\simeq & (2\bC+\Ad+\Lambda^2 \tE^*)+(\tE+\tD+3\tE^*)\tH +\\
& (3\bC+\Ad+\Lambda^2\tE^*)S^2 \tH+2\tE^* S^3\tH +S^4\tH\ .\\
\end{align*}
\vskip-0.4cm\par\noindent
Let $\imath_{11}:S^4\tH\rightarrow \wt\tE^* S^3\tH$ and 
$\imath_{12}:\Ad S^2\tH\rightarrow \tD\tH$ be the immersions  defined by
\vskip-0.15cm\par
\begin{align*}
&\imath_{11}(\odot^4 h)=(h(x)y-h(y)x)\otimes h\omega(h,r)\omega(h,s)\;,\\
&\imath_{12}(A\odot^2 h)=[A(x)h(y)-A(y)h(x)]\otimes [\omega(h,r)s+\omega(h,s)r]\;.\\
\end{align*}\vskip-0.4cm\par\noindent
The modules $\res(\Im(\imath_{11}))$ and $\res(\Im(\imath_{12}))$ are non-zero and have trivial intersection with, respectively, $\res(\Im(\imath_5))$ and $\res(\Im(\imath_4))$. Hence
\begin{align*}
B^{0,2}(\wt\gg,W)&\supset \bigoplus_{i=1}^{12} \res(\Im(\imath_i))\\
&\simeq 3\tE^*\tH+(3\bC+2\Ad)S^2\tH+(\tE+\tE^*)S^3\tH+2S^4\tH\, .
\end{align*}

(iv) We determine appropriate additional submodules 
$2\tE^*\tH$, $S^2\tH$ and $2\tE^* S^3\tH$ of $B^{0,2}(\wt\gg,W)$. We first single out $S^2\tH$. Consider the immersions of $S^2\tH$ in $\wt\tE^* \tH$
and $\wt\tD\tH$ given by $\odot^2 h\mapsto [h(x)y-h(y)x]\otimes [\omega(h,r)s+\omega(h,s)r]$ and, respectively,  $\odot^2 h\mapsto\beta_2(h)(x,y)\otimes [\omega(h,r)s+\omega(h,s)r]$ and compute their linear combination 
\begin{equation*}
\imath_{13}(\odot^2 h)=[h(x)\xi(y)-h(y)\xi(x)]\otimes [\omega(h,r)s+\omega(h,s)r]\ .
\end{equation*}
The module $\res(\Im(\imath_{13}))$ is non-zero, contained in $B^{0,2}(\wt\gg,W)$ and with trivial intersection with $\res(\Im(\imath_{3}))\oplus\res(\Im(\imath_{6}))\oplus\res(\Im(\imath_{10}))$.

We single out two submodules of type $\tE^*\tH$. Consider
$\imath_{14}:\tE^*\tH\rightarrow\wt\tE^*\tH$ and $\imath_{15}:\tE^*\tH\rightarrow\wt\tD\tH$ given by
\vskip-0.15cm\par
\begin{align*}
&\imath_{14}(e\otimes h)=(e(x)y-e(y)x)\otimes(\omega(h,r)s+\omega(h,s)r)\;,\\
&\imath_{15}(e\otimes h)=\beta_1(e)(x,y)\otimes(\omega(h,r)s+\omega(h,s)r)\;.\\
\end{align*}\vskip-0.3cm\par\noindent
Their linear combinations $\frac{2}{2n-1}(\imath_{14}-\imath_{15})$ and $\frac{2n-3}{2n-1}\imath_{14}+\frac{2}{2n-1}\imath_{15}$
determine two non-intersecting submodules of $B^{0,2}(\wt\gg,W)$ whose direct sum does also not intersect $\res(\Im(\imath_{1}))\oplus\res(\Im(\imath_{8}))\oplus\res(\Im(\imath_{9}))$.

Finally the immersions $\imath_{16}: \tE^* S^3\tH\rightarrow \wt\tE^* S^3\tH$ and $\imath_{17}: \tE^* S^3\tH\rightarrow \wt\tD\tH$ defined by
\vskip-0.15cm\par
\begin{align*}
&\imath_{16}(e\odot^3 h)=[e(x)y-e(y)x] \otimes h\omega(h,r)\omega(h,s)\,,\\
&\imath_{17}(e\odot^3 h)=\beta_3(e\odot^2h)(x,y)\otimes [\omega(h,r)s+\omega(h,s)r]\;,\\
\end{align*}\vskip-0.4cm\par\noindent
determine, together with $\res(\Im(\imath_{7}))$, three copies of $\tE^* S^3\tH$ in $B^{0,2}(\wt\gg,W)$. Summarizing
\begin{align*}
B^{0,2}(\wt\gg,W)&\supset \bigoplus_{i=1}^{17} \res(\Im(\imath_i))\\
&\simeq 
5\tE^*\tH+(4\bC+2\Ad)S^2\tH+(\tE+3\tE^*)S^3\tH+2S^4\tH\,
\end{align*}
and moreover
\begin{align*}
B^{0,2}(\wt\gg,W)\supset &(2\bC+\Ad+\Lambda^2 \tE^*+S^2\tE^*)+(\tE+\tD+\tC+5\tE^*)\tH+\phantom{cccccccc}\\
&(\Lambda^2\tE^*+S^2\tE^*+4\bC+2\Ad)S^2 \tH+(\tE+3\tE^*)S^3\tH+2S^4\tH\,,\\
H^{0,2}(\wt\gg,W)\supset& \Lambda^2\tE^* S^2\tH+(\tD+\tE^*)S^3\tH+
(\Ad+\Lambda^2\tE^*) S^4\tH+\tE^* S^5\tH\;,
\end{align*}
by equation \eqref{eq:firstsummary}. Using that $H^{0,2}(\wt\gg,W)\simeq V\otimes\Lambda^2 W^*/B^{0,2}(\wt\gg,W)$ and the decomposition of 
$V\otimes\Lambda^2 W^*$ given in Proposition \ref{primedecomp} one directly sees that both these inclusions are actually equalities.
\end{proof}
\subsection{An application of Theorem \ref{speriamobene}}
\label{esempi!}
Let $M$ be the connected, simply connected real Lie group associated with a $\bZ$-graded Lie algebra $\gm$ of the form 
$\gm=\gm_{-1}+\gm_{-2}$, where $\gm_{-1}=U=\bH^{n-1}$ and $\gm_{-2}=U^\perp=\Im(\bH)$. 
Any such $\gm=\gm_{-1}+\gm_{-2}$ is $2$-step nilpotent and therefore uniquely determined by the (real) tensor $L$ in
$
U^\perp\otimes\Lambda^2U^*\simeq\Lambda^2 \tE^*+(\Lambda^2\tE^*+S^2\tE^*)S^2\tH+\Lambda^2\tE^* S^4\tH
$
which gives the Lie bracket of two elements in $\gm_{-1}$.
We remark that from the differentiable point of view $M$ is just the $4n-1$-dimensional quaternionic Heisenberg group. In particular it admits the standard quaternionic contact structure given by the immersion as a real hypersurface of $\bH$P$^n$ and
invariant under the action of the usual Lie group multiplication on $M$. It is not difficult to see that this multiplication is determined by an element $L_o\in\Lambda^2\tE^*$.
We will now see that there exist also other homogeneous almost CR quaternionic structures on $M$, which correspond to appropriate choices of $L$ (and hence to different group multiplications on $M$). 

Fix a Lie algebra structure on $\gm$ determined by an element $L\in U^\perp\otimes\Lambda^2U^*$ and consider the natural $\{e\}$-structure (absolute parallelism) $\pi:P'\to M$ determined by the left-invariant vector fields of $M$. We define the {\it associated almost CR quaternionic structure} as
union of $G$-orbits $\pi:P=P'{\cdot}G\to M$.

The essential $(G,\wt G)$-curvature of $P$ can be easily described. Consider the global section $s:M\rightarrow P'\subset P$ determined by left-invariant vector fields
so that the associated $1$-form \eqref{MCC} given by
$$\omega^{-1}=(s^*\vartheta^1,\ldots,s^*\vartheta^{4n-1},0):TM\longrightarrow W=\gm$$
coincides with the Maurer-Cartan form of $M$. Using the Maurer-Cartan equation, the total $(G,\wt G)$-curvature is the left-invariant $\gm$-valued $2$-form
$$
\Omega^{-1}=d\omega^{-1}=-\frac{1}{2}[\omega^{-1},\omega^{-1}]\;.
$$ 
It follows that the $G$-equivariant essential curvature $\cR^1:P\to H^{0,2}(\wt\gg,W)$ is constant on $s(M)\subset P$ and naturally identifiable with the cohomology class in $H^{0,2}(\wt\gg,W)$ of $L$. Theorem \ref{thm:deformazioni} follows then from the following.
\begin{theorem}\label{seraven}
Let $\pi:P\to M$ be the almost CR quaternionic structure on the quaternionic Heisenberg group $M$ associated with $L\in U^\perp\otimes\Lambda^2 U^*$. Then
$\cR^1:P\to H^{0,2}(\wt\gg,W)$ is nonzero if and only if $L$ has a non-trivial component in $\Lambda^2 \tE^*S^2\tH+\Lambda^2\tE^*S^4\tH$ and, in this case,
$P$ is not induced by a local immersion in a quaternionic manifold and it is not quaternionic contact. Moreover
there exists a $1$-parameter family of almost CR quaternionic structures $\pi:P_t\to M$, $t\geq 0$, such that $P_t$ is isomorphic to the standard quaternionic contact structure only at $t=0$.
\end{theorem}
\begin{proof} 
By Theorem \ref{speriamobene} the unique module of type $\Lambda^2 \tE^* S^4\tH\subset V\otimes \Lambda^2 W^*$ is not included in $B^{0,2}(\wt\gg,W)$. This copy contributes therefore to $H^{0,2}(\wt\gg,W)$ and necessarily coincides with the copy $\Lambda^2 \tE^* S^4\tH\subset U^\perp\otimes\Lambda^2 U^*$.

On the other hand there are two copies of $\Lambda^2 \tE^* S^2\tH$ in $V\otimes \Lambda^2 W^*$. One is included in
$\res(\wt\tD S^3\tH)$ and contributes to cohomology, the other is included in $\res(\wt \tD\tH)\subset B^{0,2}(\wt\gg,W)$. It is not difficult to check that $\Lambda^2 \tE^* S^2\tH\subset U^\perp\otimes\Lambda^2 U^*$ has a non-trivial projection to $H^{0,2}(\wt\gg,W)$.
Finally the submodules $\Lambda^2 \tE^*$ and $S^2\tE^*S^2\tH$ do not contribute to cohomology by Theorem \ref{speriamobene}. This exhausts the four irreducible components of $U^\perp\otimes\Lambda^2 U^*$.

Our claims are a direct consequence of above observations, Theorem \ref{TeoremaRSF} and the fact 
that any quaternionic contact manifold admits an immersion in a quaternionic manifold (see \cite{Duc2}). The family is the $1$-parameter family of almost CR quaternionic structures associated to $L_t:=L_o+t L$, where $L$ is any nonzero element in $\Lambda^2 \tE^*S^2\tH+\Lambda^2 \tE^*S^4\tH$.
\end{proof}
\vskip0.3cm\par\noindent
\section{The essential \texorpdfstring{$(G,\wt G)$}--curvature \texorpdfstring{$\cR^3$}{} vanishes}
\setcounter{section}{5}
\setcounter{equation}{0}
\label{anticipoanticipo}
The main aim of this section is to improve Theorem \ref{TeoremaRSF} of \S\ref{orders} as follows.
\begin{theorem}
\label{improve}
The $G$-structure $\pi:P\to M$ is locally immersible into $\wt\pi:\wt P\to\bH\mathrm{P}^n$
if and only if it is locally immersible up to second order.
\end{theorem}
The proof of Theorem \ref{improve} is based on the following property:
{\it the generalized cohomology group $H^{2,2}(\wt\gg,W)$ vanishes}. To prove this, note that $\wt\gg_\infty^{2}=0$ and $H^{3,2}(\wt\gg,W)=0$ so that we can consider \eqref{sequenza} with $p=3$, $q=2$ and get the exact sequence
$0\longrightarrow H^{2,2}(\wt\gg,W)\longrightarrow H^{2,3}(\wt\gg)$.
Theorem \ref{improve} holds if we show $H^{2,3}(\wt\gg)=0$. 
We remark that it is more convenient to prove $H^{2,3}(\wt\gg)=0$ rather than directly showing $H^{2,2}(\wt\gg,W)=0$ with analogues of the techniques used in \S \ref{anticipo}. This has a double motivation.

On a one hand there is no need of an explicit description of the isotopic components of $H^{2,2}(\wt\gg,W)$, as we are going to show that it vanishes.
\par
On the other hand $H^{2,3}(\wt\gg)=0$ is a stronger result than $H^{2,2}(\wt\gg,W)=0$, to prove which one can exploit $\wt\gg$-equivariance (instead of simple $\gs$-equivariance) and the machinery of the Kostant version of the Borel-Bott-Weil Theorem. 
We recall here only the facts that we need in the form suitable for our purposes and refer to \cite{Kos} for more details (see also e.g. \cite{Y}).

Let $\wt\gg_\infty=\bigoplus\wt\gg_\infty^p$ be the $\bZ$-grading \eqref{maxtrans} 
of $\wt\gg_\infty=\sl_{2n+2}(\bC)$ with $\wt\gg_\infty^0=\wt\gg$ (recall also \eqref{eq:gradingprol}) and 
$E\in\wt\gg_\infty$ the associated grading element $[E,X]=p X$ for all $X\in\wt\gg_\infty^p$.
Let $\gt$, $\Phi$ be also as in \S \ref{sec:conv} and denote by $\Delta=\{\alpha_1,\dots,\alpha_\ell\}$ the simple root system of $\wt\gg_\infty$
opposite to $\{\delta_1,\dots,\delta_{\ell}\}$, i.e. with $\alpha_i=-\delta_i$ for all $i=1,\ldots,\ell$. The choice is made in such a way that $\alpha_2(E)=-1$ and $\alpha_i(E)=0$ if $i\neq 2$ (cf. \cite{Kos}). Finally we also set $\Phi_p=\left\{\alpha\in\Phi\,|\,\alpha(E)=p\right\}$.

Let $W$ be the Weyl group  of $\Phi$ and for any $\sigma\in W$ set $\Phi_{\sigma}=\Phi^+\cap\sigma(\Phi^-)$.
The {\it Hasse diagram} $W^0\subset W$ of $\wt\gg_\infty=\bigoplus\wt\gg_\infty^p$ is defined by 
$$W^0=\left\{\sigma\in W\,|\,\Phi_\sigma\cap \Phi_0=\emptyset\right\}$$ 
and the decomposition of the Spencer cohomology $H^{\bullet,\bullet}(\wt\gg)$ into $\wt\gg$-irreducible modules is given by
$$
H^{\bullet,\bullet}(\wt\gg)=\bigoplus_{\sigma\in W^0} H^\sigma\,,$$
where $H^\sigma$ is the irredubile $\wt\gg$-module with highest weight $\xi_\sigma=\sigma(\theta)-\left\langle\Phi_\sigma\right\rangle$.
Here $\theta$ is the highest root of $\wt\gg_\infty$, $\left\langle A\right\rangle=\sum_{\alpha\in A}\alpha$ for any subset $A\subset \Phi$
and weights act on $\gt$, a Cartan subalgebra also for the reductive Lie algebra $\wt\gg$.

Moreover, for any non-negative integer $q$,
$$
H^{\bullet,q}(\wt\gg)=\bigoplus_{\sigma\in W^0(q)} H^\sigma\,,
$$
where $W(q)=\{\sigma\in W\,|\,|\Phi_{\sigma}|=q\}$ and $W^0(q)=W^0\cap W(q)$. We recall that the cardinality $|\Phi_\sigma|$ of $\Phi_\sigma$ coincides with
the lenght $\ell(\sigma)$ of the reflection $\sigma\in W$ (see e.g. \cite{Humph}). Finally if 
$\sigma\in W^0(q)$ with
$\Phi_\sigma=\{\beta_1,\dots,\beta_q\}$ then  $H^\sigma\subset H^{p,q}(\wt\gg)$ where $p$ is deductible from the following identity (see \cite{Y}):
\beq
\label{3app}
\sigma(\theta)(E)=\sum_{i=1}^q \beta_i(E)+p+q-1\ .
\eeq

The following result determines $H^{\bullet,3}(\wt\gg)$. Case $p=2$ implies Theorem \ref{improve} and case $p=0$ is relevant in \S \ref{anticipoanticipoanticipo}; case $p=1$ is added only for completeness.
\begin{proposition}
\label{quelloprima}
The cohomology group $H^{2,3}(\wt\gg)$ is trivial. The groups $H^{1,3}(\wt\gg)$ and $H^{0,3}(\wt\gg)$ are $\sl_2(\bC)\oplus\sl_{2n}(\bC)$-irreducible modules isomorphic with
\begin{align*}
& \y(\ \ \ \ ,\ \ \ ) &  &\phantom{ccccccccccccccccccc}\; \y(\ \ ,\ )   \\
H^{1,3}(\wt\gg)\simeq\tH\;\otimes\;\; & \,\vdots\phantom{cccc}\vdots & \phantom{ccccccc} \text{and} \phantom{ccccccc} & H^{0,3}(\wt\gg)\simeq S^4\tH\;\otimes\;\; \vdots\phantom{i}\vdots \\
& \y(\ \ \ ,\ \ ) &  &\phantom{ccccccccccccccccccc}\;  \y(\ ) \\
\end{align*}
where the first (resp. second) Young diagram has $2n-3$ (resp. $2n-4$) rows of the form $\y(\ \ \ )$ (resp. $\y(\ )$). 
\end{proposition}
\begin{proof}
Let $\sigma_i=\sigma_{\alpha_i}$ be the reflection in $\left\langle\Phi\right\rangle_{\bR}=\operatorname{span}_{\bR}\{\alpha_1,\ldots,\alpha_\ell\}$ associated with $\alpha_i\in \Delta$ and $\sigma_{ijk}=\sigma_i{\cdot}\sigma_j{\cdot}\sigma_k$ the composition of three simple reflections. Since
\begin{align*}
\sigma_i:
\begin{cases}
\alpha_i\rightarrow -\alpha_i\\
\alpha_{i-1}\rightarrow \alpha_{i-1}+\alpha_i\\
\alpha_{i+1}\rightarrow \alpha_{i+1}+\alpha_i\\
\alpha_j\rightarrow\alpha_j\;\;\text{otherwise}
\end{cases}
\end{align*}
one gets that the set $W(3)=\left\{\sigma\in W\,|\,\ell(\sigma)=3\right\}$ of lenght 3 elements in the Weyl group  decomposes into $W(3)=W'(3)\cup W''(3)$, 
$$W'(3)=\{\sigma_{ijk}\,|\,i\neq j, j\neq k, k\neq i\}\;\;,\qquad W''(3)=\{\sigma_{iji}\,|\,i\neq j, \left\langle \alpha_i,\alpha_j\right\rangle \neq 0\}\,,$$ 
where as usual $\left\langle \alpha_i,\alpha_j\right\rangle=2(\alpha_i,\alpha_j)/(\alpha_j,\alpha_j)$ and 
$(\cdot,\cdot)$ is the positive definite scalar product on $\left\langle \Phi\right\rangle_{\bR}$ induced by the Killing form.

In order to determine the Hasse diagram $W^0(3)\subset W(3)$ we first consider $\sigma=\sigma_{ijk}\in W'(3)$ and distinguish different possible cases, depending on the mutual position of $\alpha_i,\alpha_j,\alpha_k$:
\begin{align*}
&\cdot\;\;&\begin{tikzpicture}
\node[iroot]  (2) [right=of 1] {};
\node[jroot]   (3) [right=of 2] {} edge [-] (2);
\node[kroot] (4) [right=of 3] {} edge [-] (3);
\end{tikzpicture}\quad&\text{and}\quad \Phi_{\sigma^{-1}}=\{\alpha_{i+2},\alpha_{i+1}+\alpha_{i+2},\alpha_i+\alpha_{i+1}+\alpha_{i+2}\}\,,&\phantom{ccccccccccccccccccccccccccccccccc}\\
&\cdot\;\;&\begin{tikzpicture}
\node[iroot]  (2) [right=of 1] {};
\node[kroot]   (3) [right=of 2] {} edge [-] (2);
\node[jroot] (4) [right=of 3] {} edge [-] (3);
\end{tikzpicture}\quad&\text{and}\quad \Phi_{\sigma^{-1}}=\{\alpha_{i+1},\alpha_i+\alpha_{i+1},\alpha_{i+1}+\alpha_{i+2}\}\,,&\phantom{ccccccccccccccccccccccccccccccccc}\\
&\cdot\;\;&\begin{tikzpicture}
\node[kroot]  (2) [right=of 1] {};
\node[jroot]   (3) [right=of 2] {} edge [-] (2);
\node[iroot] (4) [right=of 3] {} edge [-] (3);
\end{tikzpicture}\quad&\text{and}\quad \Phi_{\sigma^{-1}}=\{\alpha_{i-2},\alpha_{i-2}+\alpha_{i-1},\alpha_{i-2}+\alpha_{i-1}+\alpha_i\}\,,&\phantom{ccccccccccccccccccccccccccccccccc}\\
&\cdot\;\;&\begin{tikzpicture}
\node[kroot]  (2) [right=of 1] {};
\node[iroot]   (3) [right=of 2] {} edge [-] (2);
\node[jroot] (4) [right=of 3] {} edge [-] (3);
\end{tikzpicture}\quad&\text{and}\quad \Phi_{\sigma^{-1}}=\{\alpha_{i-1},\alpha_{i+1},\alpha_{i-1}+\alpha_i+\alpha_{i+1}\}\,,&\phantom{ccccccccccccccccccccccccccccccccc}\\
\end{align*}
\vskip-1.2cm
\begin{align*}
&\cdot\;\;&\begin{tikzpicture}
\node[iroot]  (2) [right=of 1] {};
\node[root] (3) [right=of 2] {} edge [-] (2);
\node[jroot]   (4) [right=of 3] {} edge [-] (3);
\node[kroot] (5) [right=of 4] {} edge [-] (4);
\end{tikzpicture}\quad&\text{and}\quad  \Phi_{\sigma^{-1}}=\{\alpha_{i},\alpha_{i+3},\alpha_{i+2}+\alpha_{i+3}\}\,,&\phantom{ccccccccccccccccccccccccccccccccccccccccccccccccccccccccccccccccccccccccccccccccccccccccccccccccccccccccccccccccccccccccccccccccccccc}\\
&\cdot\;\;&\begin{tikzpicture}
\node[iroot]  (2) [right=of 1] {};
\node[root] (3) [right=of 2] {} edge [-] (2);
\node[kroot]   (4) [right=of 3] {} edge [-] (3);
\node[jroot] (5) [right=of 4] {} edge [-] (4);
\end{tikzpicture}\quad&\text{and}\quad \Phi_{\sigma^{-1}}=\{\alpha_{i},\alpha_{i+2},\alpha_{i+2}+\alpha_{i+3}\}\,,&\phantom{ccccccccccccccccccccccccccccccccccccccccccccccccccccccccccccccccc}\\
&\cdot\;\;&\begin{tikzpicture}
\node[iroot]  (2) [right=of 1] {};
\node[jroot] (3) [right=of 2] {} edge [-] (2);
\node[root]   (4) [right=of 3] {} edge [-] (3);
\node[kroot] (5) [right=of 4] {} edge [-] (4);
\end{tikzpicture}\quad&\text{and}\quad \Phi_{\sigma^{-1}}=\{\alpha_{i+1},\alpha_{i}+\alpha_{i+1},\alpha_{i+3}\}\,,&\phantom{cccccccccccccccccccccccccccccccccccccccccccccccccccccccccccccc}\\
&\cdot\;\;&\begin{tikzpicture}
\node[jroot]  (2) [right=of 1] {};
\node[iroot] (3) [right=of 2] {} edge [-] (2);
\node[root]   (4) [right=of 3] {} edge [-] (3);
\node[kroot] (5) [right=of 4] {} edge [-] (4);
\end{tikzpicture}\quad&\text{and}\quad \Phi_{\sigma^{-1}}=\{\alpha_{i-1},\alpha_{i-1}+\alpha_{i},\alpha_{i+2}\}\,,&\phantom{cccccccccccccccccccccccccccccccccccccccccccccccccccccccccccccccccccccccccccccccccc}\\
\end{align*}
\vskip-1.2cm
\begin{align*}
&\cdot\;\;&\begin{tikzpicture}
\node[iroot]  (2) [right=of 1] {};
\node[root] (3) [right=of 2] {} edge [-] (2);
\node[]   (4) [right=of 3] {$\;\cdots\,$} edge [-] (3);
\node[root] (5) [right=of 4] {} edge [-] (4);
\node[jroot]   (6) [right=of 5] {} edge [-] (5);
\node[kroot] (7) [right=of 6] {} edge [-] (6);
\end{tikzpicture}\quad&\text{and}\quad  \Phi_{\sigma^{-1}}=\{\alpha_{i},\alpha_{j+1},\alpha_{j}+\alpha_{j+1}\}\,,&\phantom{ccccccccccccccccccccccccccccccccc}\\
&\cdot\;\;&\begin{tikzpicture}
\node[iroot]  (2) [right=of 1] {};
\node[root] (3) [right=of 2] {} edge [-] (2);
\node[]   (4) [right=of 3] {$\;\cdots\,$} edge [-] (3);
\node[root] (5) [right=of 4] {} edge [-] (4);
\node[kroot]   (6) [right=of 5] {} edge [-] (5);
\node[jroot] (7) [right=of 6] {} edge [-] (6);
\end{tikzpicture}\quad&\text{and}\quad  \Phi_{\sigma^{-1}}=\{\alpha_{i},\alpha_{j-1},\alpha_{j-1}+\alpha_{j}\}\,,&\phantom{ccccccccccccccccccccccccccccccccc}\\
&\cdot\;\;&\begin{tikzpicture}
\node[iroot]  (2) [right=of 1] {};
\node[jroot] (3) [right=of 2] {} edge [-] (2);
\node[root]   (4) [right=of 3] {} edge [-] (3);
\node[]   (5) [right=of 4] {$\;\cdots\,$} edge [-] (4);
\node[root]   (6) [right=of 5] {} edge [-] (5);
\node[kroot] (7) [right=of 6] {} edge [-] (6);
\end{tikzpicture}\quad&\text{and}\quad \Phi_{\sigma^{-1}}=\{\alpha_{i+1},\alpha_{i}+\alpha_{i+1},\alpha_{k}\}\,,&\phantom{ccccccccccccccccccccccccccccccccc}\\
&\cdot\;\;&\begin{tikzpicture}
\node[jroot]  (2) [right=of 1] {};
\node[iroot] (3) [right=of 2] {} edge [-] (2);
\node[root]   (4) [right=of 3] {} edge [-] (3);
\node[]   (5) [right=of 4] {$\;\cdots\,$} edge [-] (4);
\node[root]   (6) [right=of 5] {} edge [-] (5);
\node[kroot] (7) [right=of 6] {} edge [-] (6);
\end{tikzpicture}\quad&\text{and}\quad \Phi_{\sigma^{-1}}=\{\alpha_{i-1},\alpha_{i-1}+\alpha_{i},\alpha_{k}\}\,,&\phantom{ccccccccccccccccccccccccccccccccc}\\
\end{align*}
\vskip-1.3cm
\begin{align*}
\cdot\;\;\;\;\left\langle \alpha_i,\alpha_j\right\rangle=\left\langle \alpha_j,\alpha_k\right\rangle=\left\langle \alpha_k,\alpha_i\right\rangle=0\;\;\;\text{and}\;\;\;\Phi_{\sigma^{-1}}=\{\alpha_i,\alpha_j,\alpha_k\}\ .\phantom{cccccccccccccccccccccccccccccccccccccccccccccccccc}
\end{align*}
\vskip0.1cm\par\noindent
On the other hand if $\sigma=\sigma_{iji}\in W''(3)$ there are just two possibilities:
\begin{align*}
&\cdot\;\;&\begin{tikzpicture}
\node[iroot]  (2) [right=of 1] {};
\node[jroot]   (3) [right=of 2] {} edge [-] (2);
\end{tikzpicture}\;\;\;&\text{and}&\;\; \Phi_{\sigma^{-1}}=\{\alpha_i,\alpha_{i+1},\alpha_i+\alpha_{i+1}\}\,,&\phantom{ccccccccccccccccccccccccccccccccc}\\
&\cdot\;\;&\begin{tikzpicture}
\node[jroot]  (2) [right=of 1] {};
\node[iroot]   (3) [right=of 2] {} edge [-] (2);
\end{tikzpicture}\;\;\;&\text{and}&\;\; \Phi_{\sigma^{-1}}=\{\alpha_{i-1},\alpha_{i},\alpha_{i-1}+\alpha_{i}\}\,.&\phantom{ccccccccccccccccccccccccccccccccc}\\
\end{align*}
\vskip-0.4cm
\noindent
It follows that the roots $\sigma^{-1}\in W(3)$ satisfying 
$$\Phi_{\sigma^{-1}}\subset\Phi^+\backslash\Phi^+_0=\{\alpha=\sum_{i=1}^\ell n_i\alpha_i\in\Phi\,|\,n_i\in\bN\;\text{and}\;n_2\neq 0\}$$
are exactly
\begin{align*}
&\sigma^{-1}=\sigma_{231}\;\;\;\;\;\text{with}\;\;\; \Phi_{\sigma_{231}}=\{\alpha_2,\alpha_1+\alpha_2,\alpha_2+\alpha_3\}\;,\\
&\sigma^{-1}=\sigma_{234}\;\;\;\;\;\text{with}\;\;\; \Phi_{\sigma_{234}}=\{\alpha_2,\alpha_2+\alpha_3,\alpha_2+\alpha_3+\alpha_4\}\;.\\
\end{align*}
\vskip-0.5cm
\par\noindent
In other words we just proved that $W^0(3)=\{\sigma_{231},\sigma_{234}\}$. 

Now equation \eqref{3app} together with
$$\sigma_{231}(\theta)=\sum_{i=3}^{\ell}\alpha_i\;\;\;,\;\;\;\;\;\;\sigma_{234}(\theta)=\theta\;\;\;,$$
yields
$$
0=\sigma_{231}(\theta)(E)=-3+p+3-1\;\;\;\text{and}\;\;\;-1=\sigma_{234}(\theta)(E)=-3+p+3-1
$$
and therefore  
$$H^{2,3}(\wt\gg)=0\;\;,\;\;\;H^{1,3}(\wt\gg)\simeq H^{\sigma_{231}}\;\;,\;\;\;H^{0,3}(\wt\gg)\simeq H^{\sigma_{234}}\,,$$ 
where $H^{\sigma}$ is $\sl_2(\bC)\oplus\ggl_{2n}(\bC)$-irreducible of highest weight $\xi_{\sigma}=\sigma(\theta)-\left\langle\Phi_\sigma\right\rangle$.
A simple computation shows
 $$\xi_{\sigma_{231}}=-\alpha_1-3\alpha_2+\sum_{i=4}^{\ell}\alpha_i\;\;\;\;\text{and}\;\;\;\;\xi_{\sigma_{234}}=\alpha_1-2\alpha_2-\alpha_3+\sum_{i=5}^{\ell}\alpha_i\ .$$
Of course $\xi_{\sigma_{231}}$ and $\xi_{\sigma_{234}}$ are lowest weights for the simple system $\{\delta_1,\ldots,\delta_\ell\}$. We now wish to express $H^{\sigma_{231}}$ and $H^{\sigma_{234}}$ as highest weight modules w.r.t. the simple root system $\{\delta_1;\delta_3,\dots,\delta_\ell\}$ of the semisimple part $\sl_2(\bC)\oplus\sl_{2n}(\bC)$ of $\sl_2(\bC)\oplus\ggl_{2n}(\bC)$ (recall also \eqref{ebbenesi} and \eqref{ebbeneno}). This means that
$$(\ell-1)\delta_1+2\sum_{i=2}^{\ell}(\ell+1-i)\delta_i=0$$
and, using the longest element of the Weyl group of $\sl_2(\bC)\oplus\sl_{2n}(\bC)$, we infer that
$H^{1,3}(\wt\gg)$ and $H^{0,3}(\wt\gg)$ are standard cyclic modules with highest weights
\beq
\label{conferma}
\frac{1}{2}\delta_1+\sum_{i=3}^{\ell-1}\frac{\ell+3i-7}{\ell-1}\delta_{i}+3\frac{\ell-2}{\ell-1}\delta_\ell\,,
\eeq
and respectively
\beq
\label{conferma2}
2\delta_1+\sum_{i=3}^{\ell-2}\frac{\ell+2i-5}{\ell-1}\delta_{i}+2\frac{\ell-3}{\ell-1}\delta_{\ell-1}+\frac{\ell-3}{\ell-1}\delta_\ell\,.
\eeq
In terms of the fundamental weight $\{\omega_1\}$ of $\sl_2(\bC)$ and the fundamental weights $\{\omega_3,\dots,\omega_{\ell}\}$ of $\sl_{2n}(\bC)$, one gets that \eqref{conferma} and \eqref{conferma2} coincide with
$$
\omega_1+(\omega_{3}+\omega_{\ell-1}+2\omega_{\ell})\;\;\;\;\;\text{and}\;\;\;\;\;
4\omega_1+(\omega_3+\omega_{\ell-2})\;,
$$
which is the claim of the proposition.
\end{proof}
\vskip0.2cm\par
\section{The essential \texorpdfstring{$(G,\wt G)$}--curvature \texorpdfstring{$\cR^2:P\to H^{1,2}(\wt\gg,W)$}{}}
\setcounter{section}{6}
\setcounter{equation}{0}
\label{anticipoanticipoanticipo}
\begin{theorem}
\label{maremmabucaiola}
The natural restriction operator $\res:\wt\gg\otimes\Lambda^2 V^*\to\wt\gg\otimes\Lambda^2 W^*$ induces an $\gs$-equivariant isomorphism $H^{1,2}(\wt\gg)\simeq H^{1,2}(\wt\gg,W)$.
\end{theorem}
\begin{proof}
\par
Consider the  exact sequence
\beq
\label{totten}
0\longrightarrow H^{1,2}(\wt\gg)\longrightarrow H^{1,2}(\wt\gg,W)\longrightarrow H^{0,2}(\wt\gg,W)\longrightarrow H^{0,3}(\wt\gg)
\eeq
obtained from \eqref{sequenza} with $p=1$, $q=2$ and $H^{1,1}(\wt\gg,W)=0$ (Proposition \ref{co1ord}).
We want to show that the last map of \eqref{totten} is injective. Indeed this allows to extract from \eqref{totten} the subsequence 
$0\longrightarrow H^{1,2}(\wt\gg)\longrightarrow H^{1,2}(\wt\gg,W)\longrightarrow 0$, which is also exact amd therefore implies our claim.

We first show that there are no representation-theoretic obstructions to the existence of an $\gs$-equivariant immersion of $H^{0,2}(\wt\gg,W)$ into $H^{0,3}(\wt\gg)$. We recall that from Proposition \ref{quelloprima} we already know that $H^{0,3}(\wt\gg)\simeq S^4\tH\otimes \wt\tV$, where $\wt\tV$ is the kernel of the natural contraction from $\wt\tE\otimes\Lambda^3\wt\tE^*$ to $\Lambda^2\wt\tE^*$.

We now describe the associated branching rules from $\sl_2(\bC)\oplus\sl_{2n}(\bC)$ to the subalgebra $\gs=\sl_{2n-2}(\bC)\oplus \sl_{2}(\bC)$
according to \S\ref{sec:conv}. First of all:
\begin{align*}
\Lambda^3 \wt\tE^*&=\Lambda^3(\tE^*+\tH)=\Lambda^3\tE^*+\Lambda^2\tE^*\tH+\tE^*\,,\\
\Lambda^2\wt\tE^*&=\Lambda^2\tE^*+\tE^*\tH+\bC\;,\\
\wt\tE\otimes\Lambda^3\wt\tE^*&=\tE\Lambda^3\tE^*+\tE\Lambda^2\tE^*\tH+\tE\tE^*+\Lambda^3\tE^*\tH\\
&\;\;\;+\Lambda^2\tE^*+\Lambda^2\tE^* S^2\tH+\tE^*\tH\,.
\end{align*}
It is now convenient to distinguish between the cases $n\geq 3$ and $n=2$. If $n\geq 3$:
\begin{align*}
\wt\tE\otimes\Lambda^3\wt\tE^*&=(2\Lambda^2\tE^*+\tV+\bC+\Ad)+(2\tE^*+\tD+\Lambda^3\tE^*)\tH
+\Lambda^2\tE^* S^2\tH\;,
\end{align*}
where $\tV$ is the kernel of the natural contraction from $\tE\otimes\Lambda^3\tE^*$ to $\Lambda^2\tE^*$, so that
\begin{align*}
\wt\tV&= \wt\tE\otimes\Lambda^3\wt\tE^*/\Lambda^2\wt\tE^*=  (\Lambda^2\tE^*+\tV+\Ad)\\
&\;\;\;\;+(\tE^*+\tD+\Lambda^3\tE^*)\tH+\Lambda^2\tE^* S^2\tH\,,
\end{align*}
and finally
\begin{align*}
H^{0,3}(\wt\gg)&=  \Lambda^2\tE^*S^2\tH+(\tD+\Lambda^3\tE^*+\tE^*)S^3\tH+(\tV+\Ad+2\Lambda^2\tE^*) S^4\tH
\\ &\;\;\;+ (\tD+\Lambda^3\tE^*+\tE^*) S^5\tH+\Lambda^2 \tE^* S^6\tH\ .
\end{align*}
If $n=2$ then $\Lambda^3\tE^*=\tD=0$ and by a similar computation
$
\wt\tV=\Ad+\tE^*\tH+ S^2\tH
$ 
and $H^{0,3}(\wt\gg)= S^2\tH+\tE^* S^3\tH+(\Ad+\bC) S^4\tH+\tE^* S^5\tH+ S^6\tH$.
Note now that every $\gs$-isotipic component of $H^{0,2}(\wt\gg,W)$ 
consists of just one irreducible module, see Theorem \ref{speriamobene}, and that each of these modules is  also included in $H^{0,3}(\wt\gg)$ by the above observations. This shows that there are no obstructions for the existence of the required $\gs$-equivariant immersion.

In summary it is sufficient 
to check that the last map in \eqref{totten} acts non-trivially on just one element for each $\gs$-irreducible component of $H^{0,2}(\wt\gg,W)$. 
This can be done as in proof of Theorem \ref{quelloprima} and hence we  omit details.
\end{proof}
We conclude with some comments on Theorem \ref{maremmabucaiola}. Let $\wt M$ be a quaternionic manifold, with associated $\wt G$-structure $\wt\pi:\wt P\to \wt M$.
By definition its essential $\wt G$-curvature (= usual intrinsic torsion) $\wt\cR^1:\wt P\rightarrow H^{0,2}(\wt\gg)$ vanishes. On the other hand $\wt\cR^2:\wt P\to H^{1,2}(\wt\gg)$ is zero if and only if $\wt M$ is locally isomorphic to $\bH$P$^n$. Let also $M$ be an almost CR quaternionic manifold with corresponding $G$-structure $\pi:P\to M$ which admits a (local) immersion $\imath:M\to \wt M$. In particular $\cR^1=0$ and $\cR^2:P\to H^{1,2}(\wt\gg,W)$ is well-defined.

Theorem \ref{maremmabucaiola} says then that the restriction $\cR^2=\imath^*\wt\cR^2$ of $\wt\cR^2$ to $M$ is the only obstruction to the existence of a local immersion of $M$ into $\bH$P$^n$. We do not know if $\cR^2=0$ actually implies $\wt\cR^2=0$.

\end{document}